\documentclass[11pt]{article}
\usepackage[numbers,sort&compress]{natbib}
\usepackage{appendix,diagbox,booktabs,float}
\usepackage{enumerate,makecell}
\usepackage{tikz}
\usepackage{pgfplots}
\usepackage{amscd}
\usepackage{amsmath}
\usepackage{latexsym}
\usepackage{amsfonts}
\usepackage{amssymb}
\usepackage{amsthm}
\usepackage{bm}
\usepackage{verbatim}
\usepackage{mathrsfs}
\usepackage{enumerate}
\usepackage{hyperref}

\theoremstyle{plain}
\theoremstyle{definition}\newtheorem{theorem}{Theorem}[section]
\theoremstyle{plain}\newtheorem{lemma}[theorem]{Lemma}
\theoremstyle{plain}\newtheorem{coro}[theorem]{Corollary}
\theoremstyle{plain}
\theoremstyle{remark}\newtheorem{remark}{Remark}[section]
\usepackage{xcolor}

\newcommand{\Div}{\mathrm{div}\,}
\newcommand{\B}{\Big}

\newcommand{\be}{\begin{equation}}
\newcommand{\ee}{\end{equation}}
 \newcommand{\ba}{\begin{aligned}}
 \newcommand{\ea}{\end{aligned}}

  \newcommand{\f}{\frac}
    
  \newcommand{\ben}{\begin{enumerate}}
   \newcommand{\een}{\end{enumerate}}

\newcommand{\te}{\text}

\newcommand{\Rmnum}[1]{\expandafter\@slowromancap\romannumeral #1@}

\allowdisplaybreaks

\numberwithin{equation}{section}
\begin{document}
\title{
On  space-time derivative estimates for the fractional Navier-Stokes equations
  }
\author{Yanqing Wang\footnote{School of Mathematics and   Information Science, Zhengzhou University of Light Industry, Zhengzhou, Henan  450002,  P. R. China Email: wangyanqing20056@gmail.com},~
    Wei Wei\footnote{Corresponding author. School of Mathematics and Center for Nonlinear Studies,  Northwest University, Xi'an, Shaanxi 710127,  P. R. China.  Email: ww5998198@126.com },~~~Gang Wu \footnote{School of Mathematical Sciences, University of Chinese Academy of Sciences, Beijing 100049, P. R.
China Email: wugang2011@ucas.ac.cn} ~ and~~~Daoguo Zhou\footnote{School of Mathematics and Information Sciences, Henan Polytechnic University, Jiaozuo, Henan 454000, P. R. China Email:
zhoudaoguo@gmail.com }    }
\date{}
\maketitle
\begin{abstract}
 In this paper, we are concerned with space-time derivative estimates of solutions to the
  fractional Navier-Stokes equations.
 It is shown that  $\Lambda^{n\alpha}u^{(m)}_{t}  \in L^{\f{2(6\alpha-5)}{4m\alpha+2n\alpha+4 \alpha-5}}(0,T; L^{2}(\mathbb{R}^{3}))$  and $  \Lambda^{n }u^{(m)}_{t}  \in L^{\f{2(6\alpha-5)}{4m\alpha+2n +4 \alpha-5}}(0,T; L^{2}(\mathbb{R}^{3}))$.
This generalizes a priori bounds for the classical
Navier-Stokes system by Duff in
\cite[Acta Math. 164, 1990]{[Duff]}  and Boutros and Gibbon's  spatial  derivative estimates in \cite[Nonlinearity 37, 2024]{[BG]}. In addition, we derive that  $  u   \in L^{\f{q}{q-3}}(0,T;L^{q} (\mathbb{R}^{3}))$ with $  6\leq q\leq\infty $  and $\Lambda^{k}u   \in L^{\f{q}{ q(k+1)-3}}(0,T;L^{q} (\mathbb{R}^{3})) $ with $k\geq1, 2\leq q\leq\infty $  in the standard  Navier-Stokes equations.
  \end{abstract}

\section{Introduction}
\label{intro}
\setcounter{section}{1}\setcounter{equation}{0}
This paper addresses a priori estimates of solutions to the following  fractional Navier-Stokes equations
\be\label{GNS}
\left\{\ba
&u_{t}+ (-\Delta)^{\alpha}u+u\cdot\nabla u+\nabla p=0,~~\te{div}\,u=0,\\
&u(0)=u(0,x).\ea\right.\ee
Here, $u$ represents the velocity of fluid and $p$ denotes the pressure.
The fractional Laplacian  operator $(-\Delta)^{\alpha}$ is the infinitesimal generator of a L\'{e}vy process and defined by $\widehat{\Lambda^{2\alpha} f}(\xi)=|\xi|^{2\alpha}\hat{f}(\xi)$
,   where $\Lambda=(-\Delta)^{1/2}$ and $\hat{f}(\xi)=\frac{1}{(2\pi)^{n}}\int_{\mathbb{R}^{n}}f (x)e^{-i\xi\cdot x}\,dx$.
 The  initial data $ u_{0} $ satisfies $\Div u_{0}=0$.
The Navier-Stokes system  \eqref{GNS} with
fractional dissipation was introduced by
Lions in \cite{[Lions]}, where it was shown  that weak solutions in natural energy spaces $L^{\infty}(0,T;L^{2})\cap L^{2}(0,T;\dot{H}^{\alpha})$ are regular for $\alpha\geq5/4$.

The partial regularity of solutions to the   fractional Navier-Stokes equations \eqref{GNS}  with $\alpha<5/4$ has been established in \cite{[RWW],[CDM],[TY],[O],[JW]} .
Just as the classical Navier-Stokes equations, the full regularity of weak solutions to the fractional Navier-Stokes equations is unknown for the case  $\alpha<5/4$.
Besides the basic energy estimate $u\in L^{\infty}(0,T;L^{2}(\mathbb{R}^{3}))\cap L^{2}(0,T;\dot{H}^{1}(\mathbb{R}^{3}))$, higher derivative estimates for the Navier-Stokes equations were exploited (see \cite{[FGT],[Tao],[Constantin],[Lions],[Vasseur],[CV],[VY],[Yang],[Duff],[MPR]}). In particular,
Foias, Guillope and Temam \cite{[FGT]} obtained new a priori
estimates of the Navier-Stokes equations and proved that   Leray-Hopf solutions satisfy   $\partial^{s}_{x}  {u} \in L^{\f{2}{ 2s-1}}  (0,T;L^{2}(\mathbb{R}^{3})).$
 The estimate  $  {u} \in L^{1}(0,T;L^{\infty}(\mathbb{R}^{3})) $   due to
Tartar      was stated in  \cite{[FGT]} (see also   \cite{[Tao]}).
In a seminal work,
 Constantin   constructed   suitable weak solutions satisfying $ \nabla^{2}   {u} \in   L^{q}(0,T;L^{q}(\mathbb{T}^{3}))  $ for $q<\f43$ in
\cite{[Constantin]}. It should be pointed out that
  Constantin's second spatial derivative estimates   surpass
 the classical Stokes's estimate
 $ \nabla^{2}   {u} \in   L^{\f54}(0,T;L^{\f54}(\mathbb{T}^{3}))$.
Subsequently, Constantin's second derivative estimate  on the bounded domain was improved by
Lions     in \cite{[Lions]}, where it was shown  that $ \nabla^{2}   {u} \in   L^{\f43,\infty}(0,T;L^{\f43,\infty}(\Omega))$. For the whole space case, by the blow-up method and the galilean invariance of the transport part of the equation, Vasseur obtained  $ \nabla^{2}  {u} \in   L^{q}(0,T;L^{q}_{loc}(\mathbb{R}^{3}))  $ for $q<\f43$ in \cite{[Vasseur]}.
In \cite{[CV]}, Choi and Vasseur
addressed   the borderline  case of \cite{[Vasseur]} and
obtained the
fractional  derivative estimates    $\nabla^{\beta} u \in L^{\f{4}{\beta+1},\infty}(0,T;L^{\f{4}{\beta+1},\infty}(\mathbb{R}^{3})),  ~\beta>1.$
Recently, using the vorticity equation and   De Giorgi iteration
 Vassuer and Yang \cite{[VY]}
 refined the  previous  second derivative estimate to   $ \nabla^{2}   {u} \in   L^{\f43,\ell}(0,T;L^{\f43,\ell}(\mathbb{R}^{3}))  $ for $\ell>\f43$. In addition, the reader is referred to the latest progress by Yang in \cite{[Yang]}. There,   it is   demonstrated that $ \nabla^{n}   {u} \in   L^{p,\infty}(0,T;L^{q,\infty}(\mathbb{T}^{3})), 1/p+3/q=n+1, n\geq1, 0<p\leq q\leq\infty.  $

For higher time derivative estimates, Duff proved that  $ \partial_{ t}^{r}\partial_{x}^{s}  {u}  \in L^{\f{2}{4r+2s-1}}(0,T;L^{2}(\mathbb{R}^{3}))  $, ~$\partial_{ t}^{r}\partial_{x}^{s}  {u}  \in L^{\f{1}{2r+ s+1}}(0,T;L^{\infty}(\mathbb{R}^{3})) $ in \cite{[Duff]}. New estimates for the temporal averages of functions
containing high order time derivatives of the weak solutions to the 3D Navier-Stokes equations were presented by Chae in \cite{[Chae]}.
Next, we turn our attention back to the     fractional Navier-Stokes equations \eqref{GNS}.
A priori estimates of solutions to the  fractional Navier-Stokes system is limited. To the   knowledge  of authors, there exist only two works.
  Local regularity of weak solutions to the fractional  Navier-Stokes equations was considered by Kwon  and O\.za\'nski in \cite{[KO]}, where it was proved that
       \be
\nabla^{n}u\in L^{p,\infty}(\mathbb{R}^{3}\times(0,\infty))
,p=\f{2(3\alpha-1)}{n+2\alpha-1}, \f34<\alpha<1,  n =1,2. \ee
Boutros and Gibbon \cite{[BG]} proved that  Leray-Hopf weak solutions of the fractional Navier-Stokes equations \eqref{GNS} satisfy
\be\label{bg}
\|\Lambda^{m\alpha}u\|_{L^{2}(\mathbb{T}^{3})}\in L^{\f{2(6\alpha-5)}{2m\alpha+4\alpha-5}}, \|\Lambda^{m }u\|_{L^{2}(\mathbb{T}^{3})}\in L^{\f{2(6\alpha-5)}{2m +4\alpha-5}},  5/6<\alpha< 5/4.
\ee
The objective of this paper is to generalize Boutros and Gibbon's derivative estimate \eqref{bg}. The following is our main result.
\begin{theorem}\label{the1.1}
Assume that $u$ is a smooth solution of    the   fractional Navier-Stokes equations \eqref{GNS}. Then for each $T>0$, $m,n\in\mathbb{N}$, there holds for  $5/6<\alpha< 5/4$,
\be\label{1.1}\ba
 \Lambda^{n\alpha}u^{(m)}_{t}  \in L^{\f{2(6\alpha-5)}{4m\alpha+2n\alpha+4 \alpha-5}}(0,T; L^{2}(\mathbb{R}^{3})),~  \Lambda^{n }u^{(m)}_{t}  \in L^{\f{2(6\alpha-5)}{4m\alpha+2n +4 \alpha-5}}(0,T; L^{2}(\mathbb{R}^{3})).  \ea\ee
\end{theorem}

To prove the aforementioned theorem, we require
\begin{theorem}\label{the1.2}Assume that $u$ is a smooth solution of    the   fractional Navier-Stokes equations \eqref{GNS}. Then for  $5/6<\alpha< 5/4$, it is valid that
 \be\label{timederivative}  u _{t}^{(k )} \in L^{\f{2(6\alpha-5)}{4k\alpha+4\alpha  -5}}(0,T;L^{2} (\mathbb{R}^{3})),\ee
where $k$ is  every nonnegative integer.
\end{theorem}

We make some analysis   about above results.
Unlike  the spatial variables,
 commutator   estimates  of time variable is unknown, which yields that even the time derivative estimates  of solutions to the Navier-Stokes equations are quite involved.
This  scenario also occurs in the derivative estimates in terms of time variable for
the MHD equations (see Zheligovsky's statement in \cite{[Zheligovsky]}).
 Theorem  \ref{the1.1} and \ref{the1.2}  complement  Boutros and Gibbon's derivative estimates  \eqref{bg} of the fractional Navier-Stokes
equations. To prove these two theorems, we note that
there are two main  steps in the deduction of Duff's space-time derivative estimates of solutions to the Navier-Stokes equations in \cite{[Duff]}. There, the first step is derivation of time derivative estimates for $  u _{t}^{(k )}$ and the second one is the space-time derivative estimates for $
 \Lambda^{n }u^{(m)}_{t}$ with any   positive integer $n, m$. It seems that Theorem \ref{the1.1} is the first generalization  of  Duff's space-time derivative estimates.
Compared with the Duff's deduction in \cite{[Duff]}, the new ingredients of our proof are four-folds. Firstly,  the notations in the above two steps  are  unified and concise.  Secondly, we observe that almost all the nonlinear  estimates are the application of estimate
$$\B|\int_{\mathbb{R}^{3}}
A\cdot B\cdot Ddx \B|\leq \|A\|_{L^{\f{6}{3-2\alpha}}(\mathbb{R}^{3})}\|\nabla B\|
_{L^{\f{3}{\alpha}}(\mathbb{R}^{3})}
\|D\|_{L^{2}(\mathbb{R}^{3})}$$
 and the interpolation
$$\|\nabla B \|_{L^{\f{3}{\alpha}}(\mathbb{R}^{3})}\leq C\|\Lambda^{\alpha}B \|^{\f{6\alpha -5 }{2\alpha }}_{L^{2}(\mathbb{R}^{3})}\|\Lambda^{2\alpha} B \|_{L^{2}(\mathbb{R}^{3})}^{\f{5-4\alpha}{2\alpha }}.$$
 Thirdly, we divide the deduction of the space-time derivative estimates into four cases in turn to make the paper more readable. Fourthly, the proof presented here can be applied to the other fluid equations. For example, we will establish the space-time derivative estimates
to the MHD equations in a forthcoming paper.

To proceed further, notice that the results in  Theorem \ref{the1.1} and Theorem \ref{the1.2} are all in the context of Hilbert spaces $H^{s}$. A natural question is whether the similar estimates hold for the Sobolev space $W^{s,p}$ with $p\neq2$. In this direction, M\'alek,    Padula  and   R\r{ u}\v{z}i\v{c}ka proved that a weak solution of the Navier-Stokes equations satisfies  $\nabla  {u} \in L^{\f{q}{2q-3}}(0,T;L^{q}(\mathbb{T}^{3}))$ with $2 \leq q<\infty $, via nonlinear $p$-Laplace operator in
\cite{[MPR]}. Our next result is to extend M\'alek,    Padula  and   R\r{ u}\v{z}i\v{c}ka's derivative estimates in \cite{[MPR]}.
\begin{coro}\label{coro}
Assume that $u$ is a smooth solution to  the    standard Navier-Stokes equations \eqref{GNS} with $\alpha=1$. Then there holds
\begin{itemize}	\item[(1)] $  u   \in L^{\f{q}{q-3}}(0,T;L^{q} (\mathbb{R}^{3})),~6\leq q\leq\infty;$
 \item[(2)]  $\Lambda^{k}u   \in L^{\f{q}{ q(k+1)-3}}(0,T;L^{q} (\mathbb{R}^{3})),~k\geq1,~2\leq q\leq\infty.$
\end{itemize}
\end{coro}
The two endpoint cases $q=6$ and $q=\infty$ in the first result correspond to the regularity of Leray-Hopf weak solutions and Tartar's contribution, respectively.
The second derivative estimates  covers the
 whole space case of  a priori bounds
 due to M\'alek,    Padula  and   R\r{ u}\v{z}i\v{c}ka in \cite{[MPR]}.
 \begin{figure}[htbp]
  \centering

  \begin{minipage}{0.48\textwidth}
    \centering
\begin{tikzpicture}[scale=5,>= stealth]
\draw[->] (0,0)--(1.1,0)node[right]{\tiny{$\f1q$}};
\draw[->](0,0)--(0,1.1)node[right]{\tiny{$\f1p$}};

\draw (0.5,0)--(0.166666,0.5)node[pos=0.5,above,sloped]{\tiny weak solutions};

\draw[fill=black] (0 ,1) circle(.002);

\node[left] at ( -0.001, -0.012) {\tiny{0}};

\node[below] at (0.5,0.01) {\tiny $( \infty,2)$};
\node[right] at
(0.166666,0.5){\tiny $( 2, 6)$};
\node[right] at
(0 ,1){\tiny $(1, \infty)$};
\end{tikzpicture}
    \caption{Regularity of weak solutions and Tartar's contribution in \cite{[FGT]}}
    \label{figure1}  \end{minipage}
  \hfill
  \begin{minipage}{0.48\textwidth}
    \centering
\begin{tikzpicture}[scale=5,>= stealth]
\draw[->] (0,0)--(1.1,0)node[right]{\tiny{$\f1q$}};
\draw[->](0,0)--(0,1.1)node[right]{\tiny{$\f1p$}};

\draw (0.5,0)--(0.166666,0.5);

\draw (0 ,1)--(0.166666,0.5);

\node[left] at ( -0.001, -0.012) {\tiny{0}};

\node[below] at (0.5,0.01) {\tiny $( \infty,2)$};
\node[right] at
(0.166666,0.5){\tiny $( 2, 6)$};
\node[right] at
(0 ,1){\tiny $(1, \infty)$};
\end{tikzpicture}
    \caption{Regularity of the velocity of the standard Navier-Stokes equations in Corollary \ref{coro}}
    \label{figure2}\end{minipage}
\end{figure}
 \begin{figure}[htbp]
  \centering
  \begin{tikzpicture}[scale=1,>= stealth]
\draw[->] (0,0)--(5,0)node[right]{\tiny{$\f1q$}};
\draw[->](0,0)--(0,5.2)node[right]{\tiny{$\f1p$}};

\draw (0.5,0)--(0.166666,0.5);
\draw (0.5,0.5)--(0 , 2);
\draw (0.5,1.5)--(0 , 3);
\draw (0.5,2.5)--(0 , 4);
\draw (0 ,1)--(0.166666,0.5);
\draw [style=dashed](0 ,5)--(0.5,3.5);
\node[right] at (0.5,2.5) {\tiny $(5/2,2)$};
\node[right] at (0.5,1.5) {\tiny $( 3/2,2)$};

\node[right] at (0.5,0.5){\tiny $( 2,2)$};
\node[left] at ( -0.001, -0.012) {\tiny{0}};

\node[below] at (0.5,0.01) {\tiny $( \infty,2)$};
\node[left] at
(0 ,1){\tiny $(1, \infty)$};
\node[left] at
(0 ,2){\tiny $(1/2, \infty)$};
\node[left] at
(0 ,3){\tiny $(1/3, \infty)$};
\node[left] at
(0 ,4){\tiny $(1/4, \infty)$};
\node[above] at (0.37,0.24) {\tiny$u$};
\node[above] at (0.57,0.72) {\tiny$\nabla u$};
\node[above] at (0.57,1.92) {\tiny$\nabla^{2} u$};
\node[above] at (0.53 ,3.0) {\tiny$D^{3} u$};
\end{tikzpicture}
    \caption{Higher derivative estimates of the standard Navier-Stokes equations in Corollary \ref{coro}}
    \label{figure3}

\end{figure}
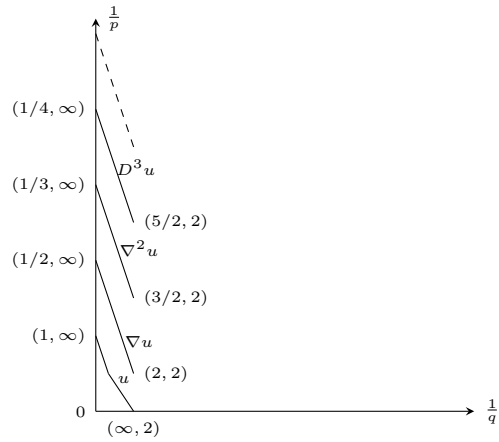

To illustrate  Corollary \ref{coro}, let  the horizontal and vertical axes be $\f1q$ and $\f1p$  on the coordinate plane,   respectively.
Regularity of Leray-Hopf weak solutions   of the tri-dimensional Navier-Stokes equations and Tartar's regularity  can be found in   Figure \ref{figure1}.
Regularity
$  u   \in L^{\f{q}{q-3}}(0,T;L^{q} (\mathbb{R}^{3})) $ with $6\leq q\leq\infty $  correspond to the line in Figure \ref{figure2}.
Higher derivative estimates  $\Lambda^{k}u   \in L^{\f{q}{ q(k+1)-3}}(0,T;L^{q} (\mathbb{R}^{3}))$ for $k\geq1,~2\leq q\leq\infty $  are
the lines in Figure \ref{figure3}. After we completed the three figures, we noticed that  a
figure similar as \ref{figure3}  appeared in work \cite{[Duff]}.

It is worth pointing out that the difference between $\Lambda$ and  $D$ is a Riesz operator. Therefore, the classical Calder\'on-Zygmund Theorem allows us to replace the operator $\Lambda$  by  $D$ in Theorem \ref{the1.1}  and Corollary \ref{coro}.

The present paper is built up as follows. Section 2 is devoted to some notations and
auxiliary lemmas. In Section 3, we study time derivative  estimates of solutions to the fractional Navier-Stokes equations. Section 4 is concerned with space-time
derivative estimates of solutions to the fractional
Navier-Stokes equations.
\section{Preliminaries}
\label{pre}
\setcounter{section}{2}\setcounter{equation}{0}
The classical Sobolev norm $\|\cdot\|_{H^{s}}$  is defined as   $\|f\|^{2} _{{H}^{s}}= \int_{\mathbb{R}^{n}} (1+|\xi|)^{2s}|\hat{f}(\xi)|^{2}d\xi$, $s\in \mathbb{R}$.
  We denote by  $ \dot{H}^{s}$ homogenous Sobolev spaces with the norm $\|f\|^{2} _{\dot{H}^{s}}= \int_{\mathbb{R}^{n}} |\xi|^{2s}|\hat{f}(\xi)|^{2}d\xi$.
 For $q\in [1,\,\infty]$, the notation $L^{q}(0,\,T;\,X)$ stands for the set of measurable functions on the interval $(0,\,T)$ with values in $X$ and $\|f(t,\cdot)\|_{X}$ belonging to $L^{q}(0,\,T)$.
For the sake of simplicity, we write
 $$  \|f\|_{L^p(\mathbb{R}^{3})}=\|f\|_{L^p}.$$
\begin{lemma}(\cite{[Duff]})\label{hofflemma}
Let $F(t)\geq0$  and $F(t)\in L^{\f{1}{p}}(0,T)$. Suppose that  $G(t)\geq0$  and
\be\label{2.1}
\f{d}{dt}F(t)+G(t)\leq CF(t)^{\f{q}{p}}.
\ee
Then, for $q>p+1$, there holds $G(t)\in L^{\f{1}{q}}(0,T)$.

\end{lemma}
\begin{proof}
Rearranging the inequality \eqref{2.1}, we have
\be
\f{\f{d}{dt}F(t)}{(F(t)+1)^{\f{q-1}{p }}}+\f{G(t)}{(F(t)+1)^{\f{q-1}{p }}}\leq C[F(t)+1]^{\f{1}{p}}.
\ee
A simple  calculation  leads to $$\int_{0}^{T}\f{G (t)}{(1+F(t))^{\f{q-1}{p }}}dt<\infty.$$
We deduce from the H\"older inequality that
$$\ba
\int_{0}^{T}G^{\f{1}{q}}(t)dt=&\int_{0}^{T}(1+F(t))^{\f{q-1}{pq}}
\f{G^{\f{1}{q}}(t)}{(1+F(t))^{\f{q-1}{pq}}}dt\\
\leq& \B[\int_{0}^{T}(1+F(t))^{\f1p}dt \B]^{\f{q-1}{ q}}\B[\int_{0}^{T}\f{G (t)}{(1+F(t))^{\f{q-1}{p }}}dt\B]^{\f1q},
\ea$$
where the H\"older inequality was used.
\end{proof}
\begin{lemma}Let  $\f{5}{2 (s+1)}<\alpha<\f54$ and $f\in \dot{H}^{\alpha}\cap\dot{H}^{s\alpha}$. Then it is valid that
\be\label{keyinequality}\|\nabla f \|_{L^{\f{3}{\alpha}}}\leq C\|\Lambda^{\alpha}f \|^{\f{2(s+1)\alpha -5 }{2\alpha(s-1)}}_{L^{2}}\|\Lambda^{s\alpha} f \|_{L^{2}}^{\f{5-4\alpha}{2\alpha(s-1)}}.\ee
\end{lemma}\begin{proof}
In view of the Bernstein inequality, we see that
\be\ba
\|\nabla f\|_{L^{\f{3}{\alpha}}}\leq& C\sum_{j<N}\|\dot{\Delta}_{j}\nabla f \|_{L^{\f{3}{\alpha}}}+C\sum_{j\geq N}\|\dot{\Delta}_{j}\nabla f \|_{L^{\f{3}{\alpha}}}\\
\leq& C\sum_{j<N}2^{j[1+3(\f12-\f{\alpha}{3})-\alpha
]}\|\dot{\Delta}_{j}\nabla f \|_{L^{2}}+C\sum_{j\geq N}2^{j[1+3(\f12-\f{\alpha}{3})-s\alpha
]}2^{js\alpha}\|\dot{\Delta}_{j}\nabla f \|_{L^{2}}.
\ea\ee
Thanks to $\alpha<\f54$ and $\alpha>\f{5}{2 (s+1)}$, we further have
\be\ba
\|\nabla f \|_{L^{\f{3}{\alpha}}}\leq& C\sum_{j<N}\|\dot{\Delta}_{j}\nabla f \|_{L^{\f{3}{\alpha}}}+C\sum_{j\geq N}\|\dot{\Delta}_{j}\nabla f \|_{L^{\f{3}{\alpha}}}\\
\leq& C 2^{N[
\f52-2\alpha
]}\|  f \|_{\dot{H}^{\alpha}}+C2^{N[
\f52-(s+1)\alpha
]}\| f \|_{\dot{H}^{s\alpha}}.
\ea\ee
To bound the right hand side of inequality, we choose $N$ such that
$$C 2^{N[
\f52-2\alpha
]}\|  f\|_{\dot{H}^{\alpha}}\approx C2^{N[
\f52-(s+1)\alpha
]}\|  f \|_{\dot{H}^{s\alpha}}.$$
Hence, we get the desired inequality \eqref{keyinequality}. The proof of this lemma is completed.
\end{proof}
Recall the homogeneous case of Kato-Ponce commutator estimates.
\begin{lemma}(\cite[Corollary 5.2, p. 68]{[Li]})\label{odelemma}
Let $\alpha>0$ and $1< p< \infty$. Then for any two Schwartz functions $f,g$ on $\mathbb{R}^{d}$, there holds
\be\label{katoponce}
\|\Lambda^{\alpha}(fg)-f\Lambda^{\alpha} g\|_{L^{p }(\mathbb{R}^{d})}\leq
C\|\nabla f\|_{L^{p_{1}  }(\mathbb{R}^{d})}\|\Lambda^{\alpha-1}g\|_{L^{p_{2}}(\mathbb{R}^{d})}
+C\|\Lambda^{\alpha} f\|_{L^{p_{3}}(\mathbb{R}^{d})}\|  g\|_{L^{p_{4} }(\mathbb{R}^{d})},
\ee
where $1<p_{1}, p_{2}, p_{3}, p_{4}\leq\infty $ and $\f{1}{p }=\f{1}{p_{1}} +\f{1}{p_{2}}= \f{1}{p_{3}} +\f{1}{p_{4}}$.
\end{lemma}
The following fractional Leibniz rule is a consequence of Lemma \ref{odelemma}.
\begin{lemma}
Let $\alpha>0$ and $\Lambda^{\alpha}f\in L^{p_{1} }(\mathbb{R}^{d}), g \in L^{p_{2} }(\mathbb{R}^{d}), f\in L^{p_{3} }(\mathbb{R}^{d}), \Lambda^{\alpha }g \in L^{p_{4} }(\mathbb{R}^{d})$.  Assume  that $1<p_{1},p_{2},p_{3},p_{4}\leq\infty$ and $1<p<\infty$. Then there holds
\be\label{fraleibnilaw}
\|\Lambda^{\alpha}(fg) \|_{L^{p }(\mathbb{R}^{d})}\leq
C\|\Lambda^{\alpha}f\|_{L^{p_{1} }(\mathbb{R}^{d})}
\|g\|_{L^{p_{2} }(\mathbb{R}^{d})}
+C\|  f\|_{L^{p_{3} }(\mathbb{R}^{d})}\|\Lambda^{\alpha }g\|_{L^{p_{4} }(\mathbb{R}^{d})},
\ee
where
\be\label{fraleibnilawcon}\f{1}{p}=\f{1}{p_{1}}+\f{1}{p_{2}}=\f{1}{p_{3}}+\f{1}{p_{4}}.
\ee\end{lemma}
\section{Time derivative estimates of solutions to the fractional Navier-Stokes equations}
\label{timeestimate}
\setcounter{section}{3}\setcounter{equation}{0}

 The purpose of this section is to prove Theorem \ref{the1.2} involving  the time derivative estimates of solutions to the fractional Navier-Stokes equations. To this end, we begin with  a series of ladder
 inequalities of solutions.

\begin{lemma}\label{lem3.1}Let $u$ satisfy   the  fractional Navier-Stokes equations
\eqref{GNS}. Then for any positive $t<T$ and $\eta$, there exists a positive constant $C$ such that
\begin{align}
&\label{3.2v1}
 \f{d}{dt}\|\Lambda^{\alpha}u (t)\|^{2}_{L^{2}}
+  \|\Lambda^{2\alpha}u (t)\|_{L^{2}}^{2}
\leq C\|\Lambda^{\alpha}u (t)\|_{L^{2}}^{\f{2(8\alpha-5)}{6\alpha-5}},
\\&\label{3.6v1}
\f{d}{dt}\|\Lambda^{\alpha}u (t)\|^{2}_{L^{2}}
+ \f{1}{2}\|\Lambda^{2\alpha}u (t)\|^{2}_{L^{2}}+\f{1}{2}\| u _{t}\|_{L^{2}}^{2}\leq C\|\Lambda^{\alpha}u (t)\|_{L^{2}}^{\f{2(8\alpha-5)}{6\alpha-5}},
\\&\label{3.13v1}
 \f{d}{dt}\| u _{t}\|^{2}_{L^{2}}+ \|\Lambda^{ \alpha}u _{t}\|_{L^{2}}^{2} \leq C\|u _{t}\|^{\f{2(10\alpha-5)}{8\alpha-5}}_{L^{2}}+C\|\Lambda^{\alpha}u \|^{\f{2(10\alpha-5)}{6\alpha-5}}_{L^{2}}+2\eta\|\Lambda^{\alpha}u \|_{L^{2}}^{\f{ 4\alpha}{6\alpha-5}}
\|\Lambda^{2\alpha}u \|^{2}_{L^{2}},
\\&\nonumber
 \f{d}{dt}\|\Lambda^{ \alpha} u _{t}\|^{2}_{L^{2}}+\f{1}{2}\|\Lambda^{2 \alpha}u _{t}\|_{L^{2}}^{2}+\f{1}{2}\| u _{tt}\|_{L^{2}}^{2}
\\&\leq  C\|\Lambda^{\alpha}u_{t}\|_{L^{2}}^{\f{2(12\alpha-5)}{10\alpha-5}}
+C\|\Lambda^{\alpha}u \|_{L^{2}}^{\f{2(12\alpha-5)}{6\alpha-5}}
+4\eta\|\Lambda^{ \alpha}u \|^{\f{8\alpha}{ 6\alpha-5}}_{L^{2}} \|\Lambda^{2 \alpha}u \|_{L^{2}}^{2}.\label{3.24v2}
\end{align}
\end{lemma}

\begin{remark}\label{remark3.1}
As a byproduct of this corollary, we have
 \begin{enumerate}[(1)]
      \item
 $\Lambda^{2\alpha}u,   u _{t}\in L^{\f{2( 6\alpha-5)}{8\alpha-5}}(0,T; L^{2});$
 \item $ \Lambda^{ \alpha}u _{t} \in L^{\f{2(6\alpha-5)}{ 10\alpha-5}}(0,T; L^{2});$
 \item$
 \Lambda^{2 \alpha}u _{t},   u _{tt} \in L^{^{\f{2(6\alpha-5)}{12\alpha-5}}}(0,T; L^{2}).$
 \end{enumerate}
\end{remark}

The proof of Lemma \ref{lem3.1} will
be divided into a sequence of subestimates as follows.
\begin{proof}
(1)
Multiplying the  fractional Navier-Stokes equations
\eqref{GNS} by $\Lambda^{2\alpha}u$, we arrive at
 \be\label{H1}
\f{1}{2}\f{d}{dt}\|\Lambda^{\alpha}u (t)\|^{2}_{L^{2}}
+ \|\Lambda^{2\alpha}u (t)\|^{2}_{L^{2}}=- \langle
u (t)\cdot\nabla u (t),\Lambda^{2\alpha}u (t)\rangle. \ee
According to the H\"older inequality, we find
$$\ba
|\langle u (t) \cdot\nabla u (t),\Lambda^{2\alpha}u (t)\rangle|
&\leq\|u (t)\|_{L^{\f{6}{3-2\alpha}}}\|\nabla u (t)\|
_{L^{\f{3}{\alpha}}}
\|\Lambda^{2\alpha}u (t)\|_{L^{2}}.
\ea$$
 By virtue of the Sobolev embedding and the
Gagliardo-Nirenberg inequality \eqref{keyinequality}, we obtain
$$
\|u (t)\|_{L^{\f{6}{3-2\alpha}} }\leq
C\|\Lambda^{\alpha}u (t)\|_{L^{2}},
$$
and
$$
\|\nabla u (t)\|_{L^{\f{3}{\alpha}} }\leq
C\|\Lambda^{\alpha}u (t)\|_{L^{2}}^{\f{6\alpha-5}{2\alpha}}
\|\Lambda^{2\alpha}u (t)\|^{\f{5-4\alpha}{2\alpha}}_{L^{2}}.
$$
This helps us to write
$$\ba
|\langle u (t) \cdot\nabla u (t),\Lambda^{2\alpha}u (t)\rangle|
 &\leq
C\|\Lambda^{\alpha}u (t)\|_{L^{2}}
\|\Lambda^{\alpha}u (t)\|_{L^{2}}^{\f{6\alpha-5}{2\alpha}}
\|\Lambda^{2\alpha}u (t)\|^{\f{5-2\alpha}{2\alpha}}_{L^{2}}
\\
&\leq
\f{1}{4} \|\Lambda^{2\alpha}u (t)\|_{L^{2}}^{2}
+C
\|\Lambda^{\alpha}u (t)\|_{L^{2}}^{\f{2(8\alpha-5)}{6\alpha-5}}
,\ea$$
 where we  used    Young's inequality.

Plugging this   into bound
\eqref{H1}, we know that
\be\label{3.2}
\f12\f{d}{dt}\|\Lambda^{\alpha}u (t)\|^{2}_{L^{2}}
+ \f34\|\Lambda^{2\alpha}u (t)\|_{L^{2}}^{2}
\leq C\|\Lambda^{\alpha}u (t)\|_{L^{2}}^{\f{2(8\alpha-5)}{6\alpha-5}}.
\ee

(2) Dotting the  fractional  Navier-Stokes equations
\eqref{GNS}  with $u _{t}$   and integrate over $\mathbb{R}^{3}$, we find
\be\label{grawall}
\f12\f{d}{dt}\|\Lambda^{\alpha}u (t)\|^{2}_{L^{2}}
+ \| u _{t}\|_{L^{2}}^{2}
=-\langle u\cdot\nabla u, u _{t}\rangle,
\ee
In view of the H\"older inequality, Sobolev inequality   and interpolation inequality \eqref{keyinequality}, we notice that
$$\ba\label{3.4}
\B|\langle u\cdot\nabla u, u _{t}\rangle\B|\leq&\|u _{t}\|_{L^{2}}\|u  \|_{L^{\f{6}{3-2\alpha}}}\|\nabla u \|_{L^{\f{3}{\alpha}}}\\\leq& C \|u _{t}\|_{L^{2}} \|\Lambda^{\alpha}u  \|_{L^{2}} \|\Lambda^{\alpha}u (t)\|_{L^{2}}^{\f{6\alpha-5}{2\alpha}}
\|\Lambda^{2\alpha}u (t)\|^{\f{5-4\alpha}{2\alpha}}_{L^{2}}
\\\leq&   \varepsilon\|u _{t}\|_{L^{2}}^{2} + C \|\Lambda^{\alpha}u (t)\|_{L^{2}}^{\f{8\alpha-5}{ \alpha}}
\|\Lambda^{2\alpha}u (t)\|^{\f{5-4\alpha}{ \alpha}}_{L^{2}}\\\leq&    \varepsilon\|u _{t}\|_{L^{2}}^{2} +C \|\Lambda^{\alpha}u (t)\|_{L^{2}}^{\f{2(8\alpha-5)}{6\alpha-5}}+
 \varepsilon\|\Lambda^{2\alpha}u (t)\|^{2} _{L^{2}},\ea
$$
where the Young inequality was used and     $\varepsilon$ will be chosen below.

Hence,
\be\ba\label{3.4v1}
\f12\f{d}{dt}\|\Lambda^{\alpha}u (t)\|^{2}_{L^{2}}
+ \| u _{t}\|_{L^{2}}^{2} \leq&    \varepsilon\|u _{t}\|_{L^{2}}^{2} +C \|\Lambda^{\alpha}u (t)\|_{L^{2}}^{\f{2(8\alpha-5)}{6\alpha-5}}+
 \varepsilon\|\Lambda^{2\alpha}u (t)\|^{2} _{L^{2}}.\ea
\ee
Adding   \eqref{3.2} and \eqref{3.4v1}, we arrive at
\be\ba
\f{d}{dt}\|\Lambda^{\alpha}u (t)\|^{2}_{L^{2}}
+  \f34\|\Lambda^{2\alpha}u (t)\|^{2}_{L^{2}}+ \| u _{t}\|_{L^{2}}^{2}\leq \varepsilon\|u _{t}\|_{L^{2}}^{2} + C\|\Lambda^{\alpha}u (t)\|_{L^{2}}^{\f{2(8\alpha-5)}{6\alpha-5}}+
 \varepsilon\|\Lambda^{2\alpha}u (t)\|^{2} _{L^{2}}.
\ea
\ee
Taking $\varepsilon$ small, we thus obtain
\be\ba\label{3.6}
\f{d}{dt}\|\Lambda^{\alpha}u (t)\|^{2}_{L^{2}}
+ \f{1}{2}\|\Lambda^{2\alpha}u (t)\|^{2}_{L^{2}}+\f{1}{2}\| u _{t}\|_{L^{2}}^{2}\leq C\|\Lambda^{\alpha}u (t)\|_{L^{2}}^{\f{2(8\alpha-5)}{6\alpha-5}}.
\ea
\ee

(3)
Differentiating the with respect to $t$, we have
\be\ba\label{3.7} u^{(2) }_{t}+\nu(-\Delta)^{\alpha}u_{t}+u_{t}\cdot\nabla u+u\cdot\nabla u_{t}+\nabla p_{t}=0,~~\te{div}\,u_{t}=0.
\ea
\ee
Multiplying  \eqref{3.7} with $u_{t}$  and taking the inner product, we discover that
\be\ba\label{3.8}
\f{1}{2}\f{d}{dt}\| u _{t}\|^{2}_{L^{2}}+\|\Lambda^{ \alpha}u _{t}\|_{L^{2}}^{2}=-\langle u_{t}\cdot\nabla u, u_{t}\rangle.
\ea
\ee
With the help of the H\"older inequality and interpolation inequality \eqref{keyinequality}, we conclude by the Young  inequality that
$$\ba
|-\langle u_{t}\cdot\nabla u, u_{t}\rangle|\leq&
\|u _{t}\|_{L^{2}}\|u_{t}  \|_{L^{\f{6}{3-2\alpha}}}\|\nabla u \|_{L^{\f{3}{\alpha}}}\\
\leq& C\|u _{t}\|_{L^{2}}\| \Lambda^{\alpha}u_{t}  \|_{L^{2}}\|\Lambda^{\alpha}u \|_{L^{2}}^{\f{6\alpha-5}{2\alpha}}
\|\Lambda^{2\alpha}u \|^{\f{5-4\alpha}{2\alpha}}_{L^{2}}\\
 \leq& \varepsilon\|\Lambda^{\alpha}u_{t}  \|_{L^{2}}^{2}+C\|u _{t}\|^{2}_{L^{2}} \|\Lambda^{\alpha}u \|_{L^{2}}^{\f{6\alpha-5}{ \alpha}}
\|\Lambda^{2\alpha}u \|^{\f{5-4\alpha}{ \alpha}}_{L^{2}}.
\ea
$$
Substituting this into \eqref{3.8}, we remark that
\be\ba\label{3.10}
\f{1}{2}\f{d}{dt}\| u _{t}\|^{2}_{L^{2}}+\f12\|\Lambda^{ \alpha}u _{t}\|_{L^{2}}^{2}  \leq&  C\|u _{t}\|^{2}_{L^{2}} \|\Lambda^{\alpha}u \|_{L^{2}}^{\f{6\alpha-5}{ \alpha}}
\|\Lambda^{2\alpha}u \|^{\f{5-4\alpha}{ \alpha}}_{L^{2}}
\\\leq&C\|u _{t}\|^{\f{2(10\alpha-5)}{8\alpha-5}}_{L^{2}}+C\|\Lambda^{\alpha}u \|^{\f{2(10\alpha-5)}{6\alpha-5}}_{L^{2}}+\eta\|\Lambda^{\alpha}u \|_{L^{2}}^{\f{ 4\alpha}{6\alpha-5}}
\|\Lambda^{2\alpha}u \|^{2}_{L^{2}},
\ea
\ee
where the Young inequality was used and $\eta$ will be determined later.\\

(4)
Dotting \eqref{3.7} with $\Lambda^{2 \alpha}u_{t}$ and integrating on $\mathbb{R}^{3}$, we arrive at
\be\ba
\f{1}{2}\f{d}{dt}\|\Lambda^{ \alpha} u _{t}\|^{2}_{L^{2}}+\|\Lambda^{2 \alpha}u _{t}\|_{L^{2}}^{2}=-\langle  u_{t}\cdot\nabla u+u\cdot\nabla u_{t},\Lambda^{2 \alpha}u_{t}\rangle.
\ea
\ee
On account of the H\"older inequality and interpolation inequality \eqref{keyinequality}, we have
\be\ba\label{3.17}
&\B|\langle u_{t}\cdot\nabla u+u\cdot\nabla u_{t}, \Lambda^{2 \alpha}u_{t} \rangle\B|\\
\leq& \|u_{t}\|_{L^{\f{6}{3-2\alpha}}} \|\nabla u \|_{L^{\f{3}{\alpha}}} \|\Lambda^{2 \alpha}u_{t}\|_{L^{2}}+\|u\|_{L^{\f{6}{3-2\alpha}}} \|\nabla u_{t} \|_{L^{\f{3}{\alpha}}}\|\Lambda^{2 \alpha}u_{t}\|_{L^{2}}\\
\leq& \|\Lambda^{\alpha}u_{t}\|_{L^{2}} \|\Lambda^{\alpha}u \|_{L^{2}}^{\f{6\alpha-5}{ 2\alpha}}
\|\Lambda^{2\alpha}u \|^{\f{5-4\alpha}{ 2\alpha}}_{L^{2}} \|\Lambda^{2 \alpha}u_{t}\|_{L^{2}}\\&+\|\Lambda^{\alpha}u\|_{L^{2}} \|\Lambda^{\alpha}u_{t} \|_{L^{2}}^{\f{6\alpha-5}{ 2\alpha}}
\|\Lambda^{2\alpha}u_{t} \|^{\f{5-4\alpha}{ 2\alpha}}_{L^{2}} \|\Lambda^{2 \alpha}u_{t}\|_{L^{2}}.
\ea
\ee
Making use of the Young inequality, we further see that
\be\ba\label{3.15}
&\|\Lambda^{\alpha}u_{t}\|_{L^{2}} \|\Lambda^{\alpha}u \|_{L^{2}}^{\f{6\alpha-5}{ 2\alpha}}
\|\Lambda^{2\alpha}u \|^{\f{5-4\alpha}{ 2\alpha}}_{L^{2}} \|\Lambda^{2 \alpha}u_{t}\|_{L^{2}}\\
\leq&  C\|\Lambda^{\alpha}u_{t}\|_{L^{2}}^{\f{2(12\alpha-5)}{10\alpha-5}}
+C\|\Lambda^{\alpha}u \|_{L^{2}}^{\f{2(12\alpha-5)}{6\alpha-5}}
+\eta\|\Lambda^{ \alpha}u \|^{\f{8\alpha}{ 6\alpha-5}}_{L^{2}} \|\Lambda^{2 \alpha}u \|_{L^{2}}^{2}+\varepsilon\|\Lambda^{2 \alpha}u_{t}\|_{L^{2}}^{2}\ea
\ee
and
\be\ba\label{3.16}
&\|\Lambda^{\alpha}u\|_{L^{2}} \|\Lambda^{\alpha}u_{t} \|_{L^{2}}^{\f{6\alpha-5}{ 2\alpha}}
\|\Lambda^{2\alpha}u_{t} \|^{\f{5-4\alpha}{ 2\alpha}}_{L^{2}} \|\Lambda^{2 \alpha}u_{t}\|_{L^{2}}\\
\leq&\|\Lambda^{\alpha}u_{t} \|_{L^{2}}^{2}\|\Lambda^{\alpha}u\|_{L^{2}}^{\f{4\alpha}{6-5\alpha}} +\varepsilon\|\Lambda^{2 \alpha}u_{t}\|_{L^{2}}^{2}
\\
\leq & C\|\Lambda^{\alpha}u_{t}\|_{L^{2}}^{\f{2(12\alpha-5)}{10\alpha-5}}
+C\|\Lambda^{\alpha}u \|_{L^{2}}^{\f{2(12\alpha-5)}{6\alpha-5}}+\varepsilon\|\Lambda^{2 \alpha}u_{t}\|_{L^{2}}^{2},
\ea
\ee
where $\eta$ and $\varepsilon$ will be chosen later.

Putting the above estimates    together and choosing $\varepsilon$ sufficiently small, we remark that
\be\ba\label{3.19v1}
 \f12\f{d}{dt}\|\Lambda^{ \alpha} u _{t}\|^{2}_{L^{2}}+\f{3}{4}\|\Lambda^{2 \alpha}u _{t}\|_{L^{2}}^{2}
&\leq  C\|\Lambda^{\alpha}u_{t}\|_{L^{2}}^{\f{2(12\alpha-5)}{10\alpha-5}}
+C\|\Lambda^{\alpha}u \|_{L^{2}}^{\f{2(12\alpha-5)}{6\alpha-5}}
+2\eta\|\Lambda^{ \alpha}u \|^{\f{8\alpha}{ 6\alpha-5}}_{L^{2}} \|\Lambda^{2 \alpha}u \|_{L^{2}}^{2}.
\ea
\ee
Multiplying both sides \eqref{3.7} of
by $u_{tt}$, we know that
\be\label{3.20}
\f12\f{d}{dt}\|\Lambda^{\alpha}u _{t}\|^{2}_{L^{2}}
+ \| u _{tt}\|_{L^{2}}^{2}
=\langle u_{t}\cdot\nabla u+u\cdot\nabla u_{t},u_{tt}\rangle.
\ee
By the same taken as \eqref{3.17}, we see that
$$\ba
&|\langle u_{t}\cdot\nabla u+u\cdot\nabla u_{t}, u_{tt}\rangle|
\\\leq& \|u_{t}\|_{L^{\f{6}{3-2\alpha}}} \|\nabla u \|_{L^{\f{3}{\alpha}}} \|u_{tt}\|_{L^{2}}+\|u\|_{L^{\f{6}{3-2\alpha}}} \|\nabla u_{t} \|_{L^{\f{3}{\alpha}}}\|u_{tt}\|_{L^{2}}\\
\leq& \|\Lambda^{\alpha}u_{t}\|_{L^{2}} \|\Lambda^{\alpha}u \|_{L^{2}}^{\f{6\alpha-5}{ 2\alpha}}
\|\Lambda^{2\alpha}u \|^{\f{5-4\alpha}{ 2\alpha}}_{L^{2}} \|u_{tt}\|_{L^{2}}+\|\Lambda^{\alpha}u\|_{L^{2}} \|\Lambda^{\alpha}u_{t} \|_{L^{2}}^{\f{6\alpha-5}{ 2\alpha}}
\|\Lambda^{2\alpha}u_{t} \|^{\f{5-4\alpha}{ 2\alpha}}_{L^{2}} \|u_{tt}\|_{L^{2}}.\ea
$$
 Replacing $\|\Lambda^{2 \alpha}u_{t}\|_{L^{2}}$ by $u_{tt}$ in \eqref{3.15} and \eqref{3.16}, we discover that
$$\ba
&\|\Lambda^{\alpha}u_{t}\|_{L^{2}} \|\Lambda^{\alpha}u \|_{L^{2}}^{\f{6\alpha-5}{ 2\alpha}}
\|\Lambda^{2\alpha}u \|^{\f{5-4\alpha}{ 2\alpha}}_{L^{2}} \|  u_{tt}\|_{L^{2}}\\
\leq&  C\|\Lambda^{\alpha}u_{t}\|_{L^{2}}^{\f{2(12\alpha-5)}{10\alpha-5}}
+C\|\Lambda^{\alpha}u \|_{L^{2}}^{\f{2(12\alpha-5)}{6\alpha-5}}
+\eta\|\Lambda^{ \alpha}u \|^{\f{8\alpha}{ 6\alpha-5}}_{L^{2}} \|\Lambda^{2 \alpha}u \|_{L^{2}}^{2}+\varepsilon\| u_{tt}\|_{L^{2}}^{2}\ea
$$
and
$$\ba
&\|\Lambda^{\alpha}u\|_{L^{2}} \|\Lambda^{\alpha}u_{t} \|_{L^{2}}^{\f{6\alpha-5}{ 2\alpha}}
\|\Lambda^{2\alpha}u_{t} \|^{\f{5-4\alpha}{ 2\alpha}}_{L^{2}} \|  u_{tt}\|_{L^{2}}\\
 \leq & C\|\Lambda^{\alpha}u_{t}\|_{L^{2}}^{\f{2(12\alpha-5)}{10\alpha-5}}
+C\|\Lambda^{\alpha}u \|_{L^{2}}^{\f{2(12\alpha-5)}{6\alpha-5}}+\varepsilon\| u_{tt}\|_{L^{2}}^{2}+\delta\| \Lambda^{2\alpha}u_{t }\|_{L^{2}}^{2},
\ea
$$
from which it follows that
$$\ba
&|\langle u_{t}\cdot\nabla u+u\cdot\nabla u_{t}, u_{tt}\rangle|
\\\leq&  C\|\Lambda^{\alpha}u_{t}\|_{L^{2}}^{\f{2(12\alpha-5)}{10\alpha-5}}
+C\|\Lambda^{\alpha}u \|_{L^{2}}^{\f{2(12\alpha-5)}{6\alpha-5}}
+\eta\|\Lambda^{ \alpha}u \|^{\f{8\alpha}{ 6\alpha-5}}_{L^{2}} \|\Lambda^{2 \alpha}u \|_{L^{2}}^{2}+2\varepsilon\| u_{tt}\|_{L^{2}}^{2}+\delta\| \Lambda^{2\alpha}u_{t }\|_{L^{2}}^{2}.
\ea
$$
where $\varepsilon$ and $\eta$ will
be specified below.

Substituting this into \eqref{3.20} and taking $\varepsilon$ sufficiently small, we see that
\be\ba\label{3.21v1}
 &\f12\f{d}{dt}\|\Lambda^{\alpha}u _{t}\|^{2}_{L^{2}}
+ \f12\| u _{tt}\|_{L^{2}}^{2}\\
\leq&  C\|\Lambda^{\alpha}u_{t}\|_{L^{2}}^{\f{2(12\alpha-5)}{10\alpha-5}}
+C\|\Lambda^{\alpha}u \|_{L^{2}}^{\f{2(12\alpha-5)}{6\alpha-5}}
+2\eta\|\Lambda^{ \alpha}u \|^{\f{8\alpha}{ 6\alpha-5}}_{L^{2}} \|\Lambda^{2 \alpha}u \|_{L^{2}}^{2}+\delta\| \Lambda^{2\alpha}u_{t }\|_{L^{2}}^{2}.
\ea
\ee
Combining \eqref{3.21v1} and \eqref{3.19v1},
taking $\delta$ small and performing some elementary estimates, we get
\be\ba\label{3.20v2}
 &\f{d}{dt}\|\Lambda^{ \alpha} u _{t}\|^{2}_{L^{2}}+\f{1}{2}\|\Lambda^{2 \alpha}u _{t}\|_{L^{2}}^{2}+ \f12\| u _{tt}\|_{L^{2}}^{2}
\\
\leq  & C\|\Lambda^{\alpha}u_{t}\|_{L^{2}}^{\f{2(12\alpha-5)}{10\alpha-5}}
+C\|\Lambda^{\alpha}u \|_{L^{2}}^{\f{2(12\alpha-5)}{6\alpha-5}}
+4\eta\|\Lambda^{ \alpha}u \|^{\f{8\alpha}{ 6\alpha-5}}_{L^{2}} \|\Lambda^{2 \alpha}u \|_{L^{2}}^{2}.
\ea
\ee
The proof of this lemma is completed.

\end{proof}

\begin{proof}[Proof of Remark \ref{remark3.1}]
\underline{$Bound$ $for$ $\|\Lambda^{2\alpha}u\|_{L^{2}}^{2}\in L^{\f{  6\alpha-5}{8\alpha-5}}.$} The basic energy estimate of the Navier-Stokes equations ensures that \be\label{3.42}
 \|\Lambda^{\alpha}u  \|_{L^{2}}^{2}\in L^1.\ee
It is a position to apply Lemma \ref{hofflemma}  to  \eqref{3.2v1} to get $\|\Lambda^{2\alpha}u (t)\|_{L^{2}}^{2}\in L^{\f{  6\alpha-5}{8\alpha-5}}(0, T).$

\underline{$Bound$ $for$ $   u _{t}\in L^{\f{2( 6\alpha-5)}{8\alpha-5}}L^{2}$.}
By virtue of Lemma \ref{hofflemma} and \eqref{3.6v1}, we conclude by $ \|\Lambda^{\alpha}u (t)\|^{2}_{L^{2}}\in L^{1}$ that $\Lambda^{2\alpha}u,   u _{t}\in L^{\f{2( 6\alpha-5)}{8\alpha-5}}L^{2}$.

\underline{$Bound$ $for$ $\|\Lambda^{ \alpha}u _{t}\|_{L^{2}}\in L^{\f{2(6\alpha-5)}{ 10\alpha-5}}$.}
To deal with the last term on the right hand side of the inequality \eqref{3.13v1}, we deduce from  \eqref{3.2} that
$$
 \f{6\alpha-5}{  8\alpha-5 } \f{d}{dt}\|\Lambda^{\alpha}u (t)\|^{\f{2(8\alpha-5)}{6\alpha-5}}_{L^{2}}
+ \|\Lambda^{\alpha}u \|^{\f{4\alpha }{6\alpha-5}}\|\Lambda^{2\alpha}u (t)\|_{L^{2}}^{2}
\leq C
\|\Lambda^{\alpha}u (t)\|_{L^{2}}^{\f{2(10\alpha-5)}{6\alpha-5}}.
$$
A combination of this and \eqref{3.10}, we end up with
\be\ba
&\f{d}{dt}\B[ \| u _{t}\|^{2}_{L^{2}}+ \f{6\alpha-5}{  8\alpha-5 }\|\Lambda^{\alpha}u (t)\|^{\f{2(8\alpha-5)}{6\alpha-5}}_{L^{2}}\B]+\f12\|\Lambda^{ \alpha}u _{t}\|_{L^{2}}^{2}  +\|\Lambda^{\alpha}u \|^{\f{4\alpha }{6\alpha-5}}\|\Lambda^{2\alpha}u (t)\|_{L^{2}}^{2}\\
\leq&  C \|u _{t}\|^{\f{2(10\alpha-5)}{8\alpha-5}}_{L^{2}}+C\|\Lambda^{\alpha}u \|^{\f{2(10\alpha-5)}{6\alpha-5}}_{L^{2}}+2\eta\|\Lambda^{\alpha}u \|_{L^{2}}^{\f{ 4\alpha}{6\alpha-5}}
\|\Lambda^{2\alpha}u \|^{2}_{L^{2}}.
\ea
\ee
By choosing $\eta$ sufficiently small, we get
\be\ba\label{3.24v1}
&\f{d}{dt}\B[ \| u _{t}\|^{2}_{L^{2}}+\f{6\alpha-5}{  8\alpha-5 }\|\Lambda^{\alpha}u (t)\|^{\f{2(8\alpha-5)}{6\alpha-5}}_{L^{2}}\B]+\f12\|\Lambda^{ \alpha}u _{t}\|_{L^{2}}^{2}    + \f12\|\Lambda^{\alpha}u \|^{\f{4\alpha }{6\alpha-5}}_{L^{2}}\|\Lambda^{2\alpha}u \|_{L^{2}}^{2} \\
\leq&  C \B[ \|u _{t}\|^{2}_{L^{2}}+\f{6\alpha-5}{  8\alpha-5 }\|\Lambda^{\alpha}u \| ^{\f{2(8\alpha-5)}{6\alpha-5}}_{L^{2}}\B]^{\f{ (10\alpha-5)}{8\alpha-5}}.
\ea
\ee
At this stage, it is enough to apply Lemma \ref{hofflemma}
to get
$\|\Lambda^{ \alpha}u _{t}\|_{L^{2}}\in L^{\f{2(6\alpha-5)}{ 10\alpha-5}}$, where $\| u _{t}\|^{2}_{L^{2}}+\|\Lambda^{\alpha}u (t)\|^{\f{2(8\alpha-5)}{6\alpha-5}}_{L^{2}}\in  L^{\f{ ( 6\alpha-5)}{8\alpha-5}}$ was used.

\underline{$Bound$ $for$ $\|\Lambda^{2 \alpha}u _{t}\|_{L^{2}}\in L^{\f{2(6\alpha-5)}{12\alpha-5}}$.}
It follows from
\eqref{3.2} that
\be\label{3.21}
\f{6\alpha-5}{ 10\alpha-5 }\f{d}{dt}\|\Lambda^{\alpha}u (t)\|^{\f{2(10\alpha-5)}{6\alpha-5}}_{L^{2}}
+ \|\Lambda^{ \alpha}u \|^{\f{8\alpha}{ 6\alpha-5}}_{L^{2}}\|\Lambda^{2\alpha}u (t)\|_{L^{2}}^{2}
\leq C
\|\Lambda^{\alpha}u (t)\|_{L^{2}}^{\f{2(12\alpha-5)}{6\alpha-5}}.
\ee
A combination of \eqref{3.20} and \eqref{3.21} yields that
\be\ba\label{3.22}
&\f{d}{dt}\B[ \|\Lambda^{ \alpha} u _{t}\|^{2}_{L^{2}}+\f{6\alpha-5}{ 10\alpha-5 }\|\Lambda^{\alpha}u (t)\|^{\f{2(10\alpha-5)}{6\alpha-5}}_{L^{2}}\B]+\f{1}{2}\|\Lambda^{2 \alpha}u _{t}\|_{L^{2}}^{2}
+\f{1}{2}\|\Lambda^{ \alpha}u \|^{\f{8\alpha}{ 6\alpha-5}}_{L^{2}}\|\Lambda^{2\alpha}u (t)\|_{L^{2}}^{2}\\&\leq  C\|\Lambda^{\alpha}u_{t}\|_{L^{2}}^{\f{2(12\alpha-5)}{10\alpha-5}}
+C\|\Lambda^{\alpha}u \|_{L^{2}}^{\f{2(12\alpha-5)}{6\alpha-5}}.
 \ea
\ee
With this in hand, combining
$\|\Lambda^{ \alpha}u _{t}\|_{L^{2}}\in L^{\f{2(6\alpha-5)}{ 10\alpha-5}}$ and Lemma \ref{hofflemma}, we find
$$
\|\Lambda^{2 \alpha}u _{t}\|_{L^{2}}^{2} \in L^{\f{6\alpha-5}{12\alpha-5}}.
$$

\underline{$Bound$ $for$
$\| u _{tt}\|_{L^{2}}\in L^{\f{2(6\alpha-5)}{12\alpha-5}}$.}
To absorb the last term on the right hand side of  $\eqref{3.21v1}$, we combine
\eqref{3.21v1} and \eqref{3.21} to obtain that
\be\ba\label{3.24}
&\f{d}{dt}\B[ \|\Lambda^{ \alpha} u _{t}\|^{2}_{L^{2}}+\f{6\alpha-5}{ 10\alpha-5 }\|\Lambda^{\alpha}u (t)\|^{\f{2(10\alpha-5)}{6\alpha-5}}_{L^{2}}\B]+\f{1}{2}\|\Lambda^{2 \alpha}u _{t}\|_{L^{2}}^{2}
\\&+\f{1}{2}\|\Lambda^{ \alpha}u \|^{\f{8\alpha}{ 6\alpha-5}}_{L^{2}}\|\Lambda^{2\alpha}u (t)\|_{L^{2}}^{2}+ \f{1}{2}\| u _{tt}\|_{L^{2}}^{2}\\\leq&  C\|\Lambda^{\alpha}u_{t}\|_{L^{2}}^{\f{2(12\alpha-5)}{10\alpha-5}}
+C\|\Lambda^{\alpha}u \|_{L^{2}}^{\f{2(12\alpha-5)}{6\alpha-5}}
\\\leq&  C\B[\|\Lambda^{\alpha}u_{t}\|^{2}_{L^{2}}
+\|\Lambda^{\alpha}u \|_{L^{2}}^{\f{2(10\alpha-5)}{6\alpha-5}} \B]^{\f{ 12\alpha-5}{10\alpha-5}}.
 \ea
\ee
At this stage, by a combination Lemma \ref{hofflemma} and $\|\Lambda^{ \alpha} u _{t}\|^{2}_{L^{2}}+\|\Lambda^{\alpha}u (t)\|^{\f{2(10\alpha-5)}{6\alpha-5}}_{L^{2}}\in    L^{\f{  6\alpha-5}{10\alpha-5}}$,
we have
$$
\|\Lambda^{2 \alpha}u _{t}\|_{L^{2}}^{2}, \| u _{tt}\|^{2}_{L^{2}}\in L^{^{\f{6\alpha-5}{12\alpha-5}}}.$$
We finish the proof of this remark.
\end{proof}

Next, we establish some energy inequalities
involving  the higher order time derivatives.
\begin{lemma}Assume that $u$ is a smooth solution of   the   fractional Navier-Stokes equations \eqref{GNS}. Then for each $t >0$, there hold
\be\ba\label{3.31v1}
 \f{d}{dt}\| u _{t}^{(r)}\|^{2}_{L^{2}}+\f{1}{2}\|\Lambda^{ \alpha}u _{t}^{(r)}\|_{L^{2}}^{2}
\leq&C\sum_{j>0}\|u_{t}^{(j)}\|
^{\f{2(4r\alpha+6\alpha-5)}{4j\alpha+4\alpha-5}}
_{L^{2}}
+C\sum_{j>0}\|\Lambda^{\alpha}u_{t}^{(r-j)} \|_{L^{2}}^{\f{2(4r\alpha+6\alpha-5)}{4(r-j)\alpha+6\alpha-5}}
\\&+2\eta\sum_{j>0}\|\Lambda^{\alpha}u_{t}^{(r-j)} \|_{L^{2}}^{\f{4\alpha(2j-1)}{  4(r-j)\alpha+6\alpha-5}}
\|\Lambda^{2\alpha}u_{t}^{(r-j)} \|^{2}_{L^{2}},
\ea
\ee
and
\be\ba\label{3.38v1}
&\f{d}{dt}\| \Lambda^{\alpha}u _{t}^{(r)}\|^{2}_{L^{2}}+\f{1}{2}\|\Lambda^{2 \alpha}u _{t}^{(r)}\|_{L^{2}}^{2} +\f{1}{2}\| u _{t}^{(r+1)}\|_{L^{2}}^{2}\\\leq& C\|\Lambda^{\alpha}u_{t}^{(r )}  \|_{L^{2}}^{\f{2(4r\alpha+8\alpha-5)}{ 4r\alpha +6\alpha-5}}+C\sum_{j=0}^{r-1}\| \Lambda^{\alpha} u_{t}^{(j)}\|_{L^{2}}^{2\f{4 r\alpha+8\alpha-5}{4j\alpha+6\alpha-5}}
\\&+C\sum_{j>0}\|\Lambda^{\alpha}u_{t}^{(r-j)} \|_{L^{2}}^{\f{2( 4 r\alpha+8\alpha-5)}{4(r-j)\alpha+6\alpha-5}}
+ \eta\sum_{j>0}\|\Lambda^{\alpha}u_{t}^{(r-j)} \|_{L^{2}}^{\f{8\alpha j}{  4(r-j)\alpha+6\alpha-5}}
\|\Lambda^{2\alpha}u_{t}^{(r-j)} \|^{2}_{L^{2}},
\ea
\ee
where  $\eta>0$ is to be chosen later.
\end{lemma}

\begin{proof}
(1)
We differentiate  the  fractional Navier-Stokes equations
\eqref{GNS} with respect to $t$ to obtain
\be\ba\label{3.26} u^{(r+1) }_{t}+\nu(-\Delta)^{\alpha}u^{(r)}_{t}+\sum_{j=0}^{r}\binom{r}{j}u_{t}^{(j)}\cdot\nabla u^{(r-j)}_{t} +\nabla p^{(r)}_{t}=0.
\ea
\ee
Taking the $L^{2}$  inner product of \eqref{3.26} with $u _{t}^{(r)}$, we find
\be\ba\label{3.27}
\f{1}{2}\f{d}{dt}\| u _{t}^{(r)}\|^{2}_{L^{2}}+\|\Lambda^{ \alpha}u _{t}^{(r)}\|_{L^{2}}^{2}=\sum_{j=0}^{r}\binom{r}{j}\langle u_{t}^{(j)}\cdot\nabla u^{(r-j)}_{t}, u _{t}^{(r)}\rangle.
\ea
\ee
Making full use of divergence free condition, we get
$$\langle u \cdot\nabla u^{(r )}_{t}, u _{t}^{(r)}\rangle=0.
$$
Hence, we rewrite \eqref{3.27} as
\be\ba\label{3.29}
\f{1}{2}\f{d}{dt}\| u _{t}^{(r)}\|^{2}_{L^{2}}+\|\Lambda^{ \alpha}u _{t}^{(r)}\|_{L^{2}}^{2}=\sum_{j>0}^{r}\binom{r}{j}\langle u_{t}^{(j)}\cdot\nabla u_{t}^{(r-j)}, u _{t}^{(r)}\rangle.
\ea
\ee
Following the same path of \eqref{3.17}, we conclude by the Young inequality that
\be\ba\label{3.30}
\langle u_{t}^{(j)}\cdot\nabla u_{t}^{(r-j)}, u _{t}^{(r)}\rangle
\leq& \|u_{t}^{(j)}\|_{L^{2}} \|\nabla u^{(r-j)}_{t}\|_{L^{\f{3}{\alpha}}}\| u _{t}^{(r)}\|_{L^{\f{6}{3-2\alpha}}}\\
\leq&\|u_{t}^{(j)}\|_{L^{2}}\|\Lambda^{\alpha}u_{t}^{(r-j)} \|_{L^{2}}^{\f{6\alpha-5}{ 2\alpha}}
\|\Lambda^{2\alpha}u_{t}^{(r-j)} \|^{\f{5-4\alpha}{ 2\alpha}}_{L^{2}}    \|\Lambda^{\alpha} u _{t}^{(r)}\|_{L^{2}} \\
\leq&C\|u_{t}^{(j)}\|^{2}_{L^{2}}\|\Lambda^{\alpha}u_{t}^{(r-j)} \|_{L^{2}}^{\f{6\alpha-5}{  \alpha}}
\|\Lambda^{2\alpha}u_{t}^{(r-j)} \|^{\f{5-4\alpha}{ \alpha}}_{L^{2}}  +\varepsilon  \|\Lambda^{\alpha} u _{t}^{(r)}\|_{L^{2}}^{2},\ea
\ee
where $\varepsilon$ satisfy $
2\sum\limits_{j>0}^{r}\binom{r}{j}\varepsilon<1/2$.

Inserting \eqref{3.30} into \eqref{3.29} and using the Young inequality once again, we discover that
 \be\ba\label{3.31}
 \f{d}{dt}\| u _{t}^{(r)}\|^{2}_{L^{2}}+\f{1}{2}\|\Lambda^{ \alpha}u _{t}^{(r)}\|_{L^{2}}^{2}
\leq&C\sum_{j>0}\|u_{t}^{(j)}\|
^{\f{2(4r\alpha+6\alpha-5)}{4j\alpha+4\alpha-5}}
_{L^{2}}
+C\sum_{j>0}\|\Lambda^{\alpha}u_{t}^{(r-j)} \|_{L^{2}}^{\f{2(4r\alpha+6\alpha-5)}{4(r-j)\alpha+6\alpha-5}}
\\&+2\eta\sum_{j>0}\|\Lambda^{\alpha}u_{t}^{(r-j)} \|_{L^{2}}^{\f{4\alpha(2j-1)}{  4(r-j)\alpha+6\alpha-5}}
\|\Lambda^{2\alpha}u_{t}^{(r-j)} \|^{2}_{L^{2}},
\ea
\ee
where $\eta$ will be fixed later.  \\

(2) It remains to show \eqref{3.38v1}. To do this,
dotting \eqref{3.26} by
$\Lambda^{2\alpha}u _{t}^{(r)}$ and integrating, we see that
\be\ba\label{3.32}
\f{1}{2}\f{d}{dt}\| \Lambda^{\alpha}u _{t}^{(r)}\|^{2}_{L^{2}}+\|\Lambda^{2 \alpha}u _{t}^{(r)}\|_{L^{2}}^{2}=\sum_{j=0}^{r}\binom{r}{j}\int u_{t}^{(j)}\cdot\nabla u^{(r-j)}_{t}\cdot \Lambda^{2\alpha}u _{t}^{(r)}dx.
\ea
\ee
Along the line of \eqref{3.17}, we arrive at
\be\ba
&\sum_{j=0}|\langle u_{t}^{(j)}\cdot\nabla u^{(r-j)}_{t}\cdot \Lambda^{2\alpha}u _{t}^{(r)}\rangle|\\
\leq& \sum_{j=0}\| u_{t}^{(j)}\|_{L^{\f{6}{3-2\alpha}}}\|\nabla u^{(r-j)}_{t}\|_{L^{\f{3}{\alpha}}}\|\Lambda^{2\alpha}u _{t}^{(r)}\|_{L^{2}}\\
\leq& \sum_{j=0}\| \Lambda^{\alpha} u_{t}^{(j)}\|_{L^{2}}\|\Lambda^{\alpha}u_{t}^{(r-j)} \|_{L^{2}}^{\f{6\alpha-5}{ 2\alpha}}
\|\Lambda^{2\alpha}u_{t}^{(r-j)} \|^{\f{5-4\alpha}{ 2\alpha}}_{L^{2}}\|\Lambda^{2\alpha}u _{t}^{(r)}\|_{L^{2}}\\
\leq& \sum_{j>0}\| \Lambda^{\alpha} u_{t}^{(j)}\|_{L^{2}}^{2}\|\Lambda^{\alpha}u_{t}^{(r-j)} \|_{L^{2}}^{\f{6\alpha-5}{  \alpha}}
\|\Lambda^{2 \alpha}u_{t}^{(r-j)} \|^{\f{5-4\alpha}{ \alpha}}_{L^{2}}+\f\delta2\|\Lambda^{2\alpha}u _{t}^{(r)}\|_{L^{2}}^{2}\\&+\| \Lambda^{\alpha} u \|_{L^{2}}\|\Lambda^{\alpha}u_{t}^{(r )} \|_{L^{2}}^{\f{6\alpha-5}{ 2\alpha}}
\|\Lambda^{2\alpha}u_{t}^{(r )} \|^{\f{5-4\alpha}{ 2\alpha}}_{L^{2}}\|\Lambda^{2\alpha}u _{t}^{(r)}\|_{L^{2}}\\
\leq& \sum_{j>0}\| \Lambda^{\alpha} u_{t}^{(j)}\|_{L^{2}}^{2}\|\Lambda^{\alpha}u_{t}^{(r-j)} \|_{L^{2}}^{\f{6\alpha-5}{  \alpha}}
\|\Lambda^{ 2\alpha}u_{t}^{(r-j)} \|^{\f{5-4\alpha}{ \alpha}}_{L^{2}}+\delta\|\Lambda^{2\alpha}u _{t}^{(r)}\|_{L^{2}}^{2}\\&+C\B[\| \Lambda^{\alpha} u \|_{L^{2}}\|\Lambda^{\alpha}u_{t}^{(r )} \|_{L^{2}}^{\f{6\alpha-5}{ 2\alpha}}\B]^{\f{4\alpha}{6\alpha-5}}.
 \ea
\ee
where we will choose $\delta$ such that  $2\delta<1/2.$

The Young inequality implies that
\be\ba
  \B[\| \Lambda^{\alpha} u \|_{L^{2}}\|\Lambda^{\alpha}u_{t}^{(r )} \|_{L^{2}}^{\f{6\alpha-5}{ 2\alpha}} \B]^{\f{4\alpha}{6\alpha-5}}\leq C\|\Lambda^{\alpha}u  \|_{L^{2}}^{\f{2(4r\alpha+8\alpha-5)}{  6\alpha-5}}+C\|\Lambda^{\alpha}u_{t}^{(r )}  \|_{L^{2}}^{\f{2(4r\alpha+8\alpha-5)}{ 4r\alpha +6\alpha-5}}
\ea
\ee
and
\be\ba\label{3.34}
&\sum_{j>0}\| \Lambda^{\alpha} u_{t}^{(j)}\|_{L^{2}}^{2}\|\Lambda^{\alpha}u_{t}^{(r-j)} \|_{L^{2}}^{\f{6\alpha-5}{ 2\alpha}}
\|\Lambda^{ 2\alpha}u_{t}^{(r-j)} \|^{\f{5-4\alpha}{ \alpha}}_{L^{2}}\\
\leq&C\| \Lambda^{\alpha} u_{t}^{(j)}\|_{L^{2}}^{\f{2 (4 r\alpha+8\alpha-5)}{4j\alpha+6\alpha-5}}
+C\|\Lambda^{\alpha}u_{t}^{(r-j)} \|_{L^{2}}^{\f{2(4 r\alpha+8\alpha-5)}{4(r-j)\alpha+6\alpha-5}}
\\&+\eta\sum_{j>0}\|\Lambda^{\alpha}u_{t}^{(r-j)} \|_{L^{2}}^{\f{8\alpha j}{  4(r-j)\alpha+6\alpha-5}}
\|\Lambda^{2\alpha}u_{t}^{(r-j)} \|^{2}_{L^{2}},
 \ea
\ee
where $\eta$ will be determined later.

From \eqref{3.32}-\eqref{3.34}, we end up with
\be\ba\label{3.35}
&\f{d}{dt}\| \Lambda^{\alpha}u _{t}^{(r)}\|^{2}_{L^{2}}+\f{1}{2}\|\Lambda^{2 \alpha}u _{t}^{(r)}\|_{L^{2}}^{2}\\\leq& C\|\Lambda^{\alpha}u  \|_{L^{2}}^{\f{2(4r\alpha+8\alpha-5)}{  6\alpha-5}}+C\|\Lambda^{\alpha}u_{t}^{(r )}  \|_{L^{2}}^{\f{2(4r\alpha+8\alpha-5)}{ 4r\alpha +6\alpha-5}}+C\| \Lambda^{\alpha} u_{t}^{(j)}\|_{L^{2}}^{2\f{4 r\alpha+8\alpha-5}{4j\alpha+6\alpha-5}}
\\&+C\|\Lambda^{\alpha}u_{t}^{(r-j)} \|_{L^{2}}^{\f{2( 4 r\alpha+8\alpha-5)}{4(r-j)\alpha+6\alpha-5}}
+\eta\sum_{j>0}\|\Lambda^{\alpha}u_{t}^{(r-j)} \|_{L^{2}}^{\f{8\alpha j}{  4(r-j)\alpha+6\alpha-5}}
\|\Lambda^{2\alpha}u_{t}^{(r-j)} \|^{2}_{L^{2}}.
\ea
\ee
We multiply equations
\eqref{3.26} by  $u _{t}^{(r+1)}$ and integrate to
obtain
\be\ba
 \f{1}{2}\f{d}{dt}\| \Lambda^{\alpha}u _{t}^{(r)}\|^{2}_{L^{2}}+\| u _{t}^{(r+1)}\|_{L^{2}}^{2}=&\sum_{j=0}^{r}\binom{r}{j}\int u_{t}^{(j)}\cdot\nabla u^{(r-j)}_{t}\cdot u _{t}^{(r+1)}dx\\
\leq& C\sum_{j=0}\| u_{t}^{(j)}\|_{L^{\f{6}{3-2\alpha}}}\|\nabla u^{(r-j)}_{t}\|_{L^{\f{3}{\alpha}}}\| u _{t}^{(r+1)}\|_{L^{2}}\\
\leq& C\sum_{j=0}\| \Lambda^{\alpha} u_{t}^{(j)}\|_{L^{2}}\|\Lambda^{\alpha}u_{t}^{(r-j)} \|_{L^{2}}^{\f{6\alpha-5}{ 2\alpha}}
\|\Lambda^{2\alpha}u_{t}^{(r-j)} \|^{\f{5-4\alpha}{ 2\alpha}}_{L^{2}} \| u _{t}^{(r+1)}\|_{L^{2}}.
\ea
\ee
The Young inequality ensures that
\be\ba& \f{d}{dt}\| \Lambda^{\alpha}u _{t}^{(r)}\|^{2}_{L^{2}}+\| u _{t}^{(r+1)}\|_{L^{2}}^{2}\\
\leq& C\sum_{j>0}\| \Lambda^{\alpha} u_{t}^{(j)}\|_{L^{2}}^{2}\|\Lambda^{\alpha}u_{t}^{(r-j)} \|_{L^{2}}^{\f{6\alpha-5}{ 2\alpha}}
\|\Lambda^{ 2\alpha}u_{t}^{(r-j)} \|^{\f{5-4\alpha}{ \alpha}}_{L^{2}}+2\varepsilon\| u _{t}^{(r+1)}\|_{L^{2}}^{2}\\&+2\| \Lambda^{\alpha} u \|_{L^{2}}\|\Lambda^{\alpha}u_{t}^{(r )} \|_{L^{2}}^{\f{6\alpha-5}{ 2\alpha}}
\|\Lambda^{2\alpha}u_{t}^{(r )} \|^{\f{5-4\alpha}{ 2\alpha}}_{L^{2}}\| u _{t}^{(r+1)}\|_{L^{2}}\\
\leq& C\sum_{j>0}\| \Lambda^{\alpha} u_{t}^{(j)}\|_{L^{2}}^{2}\|\Lambda^{\alpha}u_{t}^{(r-j)} \|_{L^{2}}^{\f{6\alpha-5}{ 2\alpha}}
\|\Lambda^{ 2\alpha}u_{t}^{(r-j)} \|^{\f{5-4\alpha}{ \alpha}}_{L^{2}}\\&+ \delta\|\Lambda^{2\alpha}u _{t}^{(r)}\|_{L^{2}}^{2}+2\varepsilon\| u _{t}^{(r+1)}\|_{L^{2}}^{2}+C\B[\| \Lambda^{\alpha} u \|_{L^{2}}\|\Lambda^{\alpha}u_{t}^{(r )} \|_{L^{2}}^{\f{6\alpha-5}{ 2\alpha}}\B]^{\f{4\alpha}{6\alpha-5}},
\ea
\ee
where $\delta$ and $\varepsilon$ are to be determined later.

As a consequence, we choose $\varepsilon$ sufficient small to get
\be\ba\label{3.37}
&\f{d}{dt}\| \Lambda^{\alpha}u _{t}^{(r)}\|^{2}_{L^{2}}+\f{1}{2}\| u _{t}^{(r+1)}\|_{L^{2}}^{2}\\\leq& C\|\Lambda^{\alpha}u  \|_{L^{2}}^{\f{2(4r\alpha+8\alpha-5)}{  6\alpha-5}}+C\|\Lambda^{\alpha}u_{t}^{(r )}  \|_{L^{2}}^{\f{2(4r\alpha+8\alpha-5)}{ 4r\alpha +6\alpha-5}}+C\| \Lambda^{\alpha} u_{t}^{(j)}\|_{L^{2}}^{2\f{4 r\alpha+8\alpha-5}{4j\alpha+6\alpha-5}}
\\&+C\|\Lambda^{\alpha}u_{t}^{(r-j)} \|_{L^{2}}^{\f{2( 4 r\alpha+8\alpha-5)}{4(r-j)\alpha+6\alpha-5}}
+\eta\sum_{j>0}\|\Lambda^{\alpha}u_{t}^{(r-j)} \|_{L^{2}}^{\f{8\alpha j}{  4(r-j)\alpha+6\alpha-5}}
\|\Lambda^{2\alpha}u_{t}^{(r-j)} \|^{2}_{L^{2}}+\delta\|\Lambda^{2\alpha}u _{t}^{(r)}\|_{L^{2}}^{2},
\ea
\ee
where the Young inequality was used and $\eta$ will be fixed later.\\
Combining \eqref{3.37} and
\eqref{3.35}, we discover that
\be\ba\label{3.38}
&\f{d}{dt}\| \Lambda^{\alpha}u _{t}^{(r)}\|^{2}_{L^{2}}+\f{1}{2}\|\Lambda^{2 \alpha}u _{t}^{(r)}\|_{L^{2}}^{2} +\f{1}{2}\| u _{t}^{(r+1)}\|_{L^{2}}^{2}\\\leq& C\|\Lambda^{\alpha}u  \|_{L^{2}}^{\f{2(4r\alpha+8\alpha-5)}{  6\alpha-5}}+C\|\Lambda^{\alpha}u_{t}^{(r )}  \|_{L^{2}}^{\f{2(4r\alpha+8\alpha-5)}{ 4r\alpha +6\alpha-5}}+C\sum_{j>0}\| \Lambda^{\alpha} u_{t}^{(j)}\|_{L^{2}}^{2\f{4 r\alpha+8\alpha-5}{4j\alpha+6\alpha-5}}
\\&+C\sum_{j>0}\|\Lambda^{\alpha}u_{t}^{(r-j)} \|_{L^{2}}^{\f{2( 4 r\alpha+8\alpha-5)}{4(r-j)\alpha+6\alpha-5}}
+ \eta\sum_{j>0}\|\Lambda^{\alpha}u_{t}^{(r-j)} \|_{L^{2}}^{\f{8\alpha j}{  4(r-j)\alpha+6\alpha-5}}
\|\Lambda^{2\alpha}u_{t}^{(r-j)} \|^{2}_{L^{2}}.
\ea
\ee
This proves the lemma.
\end{proof}
It is a position to show Theorem \ref{the1.2}.
\begin{proof}[Proof of Theorem \ref{the1.2}]
We assert that the following fact is valid for any $r=k\in \mathbb{N}$,
\be\label{3.40}
\left\{\ba
&F_{r,0}\in L^{\f{6\alpha-5}{4r\alpha+4\alpha-5}},\\
&\f{d}{dt}F_{r,0}+G_{r,0}\leq CF^{\f{4r\alpha+6 \alpha-5}{4r\alpha+4 \alpha-5}}_{r,0},\\
&G_{r,0}\in L^{\f{6\alpha-5}{4r\alpha+6 \alpha-5}}, \|\Lambda^{\alpha}u^{( r) }_{t}\|_{L^{2}}^{2}\in L^{\f{6\alpha-5}{4r\alpha+6 \alpha-5}},\\
&F_{r,1}\in L^{\f{6\alpha-5}{4r\alpha+6\alpha-5}},\\
&\f{d}{dt}F_{r,1}+G_{r,1}\leq CF^{\f{4r\alpha+7 \alpha-5}{4r\alpha+6 \alpha-5}}_{r,1},\\
&G_{r,1}\in L^{\f{6\alpha-5}{4r\alpha+8 \alpha-5}}, \|\Lambda^{2\alpha}u^{( r) }_{t}\|_{L^{2}}^{2},\| u^{( r+1) }_{t}\|_{L^{2}}^{2}\in L^{\f{6\alpha-5}{4r\alpha+4 \alpha-5}},\ea\right.\ee
where
\be\ba\label{3.41}
&F_{r,0}=\|u_{t}^{(r)}\|_{L^{2}}^{2}+\sum_{j=1}^{r-1}\f{4j\alpha+4 \alpha-5}{4r\alpha+4 \alpha-5}F_{j,0}^{\f{4r\alpha+4 \alpha-5}{4j\alpha+4 \alpha-5}}+\sum_{j=0}^{r-1}\f{4j\alpha+6 \alpha-5}{4r\alpha+4 \alpha-5}F_{j,1}^{\f{4r\alpha+4 \alpha-5}{4j\alpha+6 \alpha-5}},\\
&G_{r,0}=\f12\|\Lambda^{\alpha}u_{t}^{(r)}\|_{L^{2}}^{2}+\f12
\sum_{j=1}^{r-1}F_{j,1}^{\f{4(r-j)\alpha-2\alpha}{4j\alpha+6 \alpha-5}}G_{j,1},\\
&F_{r,1}=\|\Lambda^{\alpha}u_{t}^{(r)}\|_{L^{2}}^{2}+\sum_{j=0}^{r-1}\f{4j\alpha+6 \alpha-5}{4r\alpha+6 \alpha-5}F_{j,1}^{\f{4r\alpha+6 \alpha-5}{4j\alpha+6 \alpha-5}},\\
&G_{r,1}=\f12\|\Lambda^{2\alpha}u_{t}^{(r)}\|_{L^{2}}^{2}+\f12\| u_{t}^{(r+1)}\|_{L^{2}}^{2}+\f12
\sum_{j=0}^{r-1}F_{j,1}^{\f{4(r-j)\alpha-2\alpha}{4j\alpha+6 \alpha-5}}G_{j,1}.
\ea\ee
Roughly speaking,
inequalities $\eqref{3.40}_{2}$ and $\eqref{3.40}_{5}$ correspond to energy inequalities \eqref{3.31v1} and \eqref{3.38v1}, respectively. The notations $F_{r,0}, G_{r,0}, F_{r,1}$ and $ G_{r,1}$ are also based on energy inequalities \eqref{3.31v1} and \eqref{3.38v1}.

Next, we will show \eqref{3.40}
for case $r=1$. For the reader's convenience, we have
\be\ba\label{3.39} &F_{0,1}=\|\Lambda^{\alpha}u  \|_{L^{2}}^{2}, \\&G_{0,1}=\f12\|\Lambda^{2\alpha}u  \|_{L^{2}}^{2}+\f12\| u_{t}  \|_{L^{2}}^{2}, \\
&F_{1,0}=\| u_{t}  \|_{L^{2}}^{2}+\f{6\alpha-5}{ 8\alpha-5 }\|\Lambda^{\alpha}u  \|_{L^{2}}^{\f{2(8\alpha-5)}{6\alpha-5}}=\| u_{t}  \|_{L^{2}}^{2}+\f{6\alpha-5}{ 8\alpha-5 }F_{0,1}^{\f{ (8\alpha-5)}{6\alpha-5}},\\
&G_{1,0}=\f12\| \Lambda^{\alpha} u_{t}  \|_{L^{2}}^{2}+\f12F_{0,1}^{\f{2\alpha}{6\alpha-5}}G_{0,1}, \\
&F_{1,1}=\| \Lambda^{\alpha} u_{t}  \|_{L^{2}}^{2}+\f{6\alpha-5}{10\alpha-5}F_{0,1}^{\f{10\alpha-5}{6\alpha-5}},\\
&G_{1,1}=\f12\| \Lambda^{2\alpha} u_{t}  \|_{L^{2}}^{2}+\f12\| u_{tt}  \|_{L^{2}}^{2}+\f12F_{0,1}^{\f{4\alpha }{6\alpha-5}}G_{0,1} +\f12F_{1,0}^{\f{4\alpha }{8\alpha-5}}G_{1,0}.\ea
\ee
We deduce from  \eqref{3.6v1}  in Lemma \ref{lem3.1} that
  \be\label{3.43v1}\f{d}{dt}F_{0,1}+G_{0,1}\leq CF_{0,1}^{\f{ (8\alpha-5)}{6\alpha-5}}.
\ee
In the light of
Lemma \ref{hofflemma}, we conclude by \eqref{3.43v1} and \eqref{3.42} that
$G_{0,1}\in L^{\f{6\alpha-5}{8\alpha-5}}$. Hence, we further have
\be\label{3.43}
F_{ 1,0}=\| u_{t}  \|_{L^{2}}^{2}+\f{6\alpha-5}{ 8\alpha-5 }\|\Lambda^{\alpha}u  \|_{L^{2}}^{\f{2(8\alpha-5)}{6\alpha-5}}\in   L^{\f{ 6\alpha-5 }{ 8\alpha-5}}.
\ee
Invoking the estimate \eqref{3.24v1}, we observe that
\be\label{3.45}
\f{d}{dt}F_{ 1,0}+G_{1,0}\leq CF_{ 1,0}^{\f{ (10\alpha-5)}{8\alpha-5}}.
\ee
  From this and \eqref{3.43}, we derive Lemma \ref{hofflemma} that
 \be\ba
G_{1,0}\in L^{\f{6\alpha-5}{10 \alpha-5}},
\ea
\ee
which turns out
 \be\ba\label{3.48}
\|\Lambda^{\alpha}u  _{t}\|_{L^{2}}^{2}\in L^{\f{6\alpha-5}{10\alpha-5}}.
\ea
\ee
We deduce from
\eqref{3.42}, \eqref{3.43} and \eqref{3.48} that
\be\label{3.49}
F_{1,1}\in L^{\f{6\alpha-5}{10\alpha-5}}.
\ee
By carrying out the necessary computation, we conclude by
\eqref{3.43v1}     that
\be\ba
\f{6\alpha-5}{10\alpha-5}\f{d}{dt}F_{0,1}^{\f{10\alpha-5}{6\alpha-5}}+F_{0,1}^{\f{4\alpha }{6\alpha-5}}G_{0,1} \leq CF_{0,1}^{\f{ (12\alpha-5)}{6\alpha-5}}.
\ea\ee
Combining this and \eqref{3.24v1}, we remark that
\be\ba \label{3.50v1}
&\f{d}{dt}\B[\|\Lambda^{ \alpha} u _{t}\|^{2}_{L^{2}}+\f{6\alpha-5}{10\alpha-5}F_{0,1}^{\f{10\alpha-5}{6\alpha-5}}\B]
\\&+\f{1}{2}\|\Lambda^{2 \alpha}u _{t}\|_{L^{2}}^{2}+\f{1}{2}\| u_{tt} \|_{L^{2}}^{2}
+ F_{0,1}^{\f{4\alpha }{6\alpha-5}}G_{0,1} \\
\leq&  C\|\Lambda^{\alpha}u_{t}\|_{L^{2}}^{\f{2(12\alpha-5)}{10\alpha-5}}
+C\|\Lambda^{\alpha}u \|_{L^{2}}^{\f{2(12\alpha-5)}{6\alpha-5}}
+2\eta\|\Lambda^{ \alpha}u \|^{\f{8\alpha}{ 6\alpha-5}}_{L^{2}} \|\Lambda^{2 \alpha}u \|_{L^{2}}^{2}+CF_{0,1}^{\f{ (12\alpha-5)}{6\alpha-5}}.
\ea
\ee
According to the definition of $F_{1,1}$ in \eqref{3.39}, we obtain
$$F_{0,1}^{\f{ (12\alpha-5)}{6\alpha-5}}\leq CF_{1,1}^{\f{ (12\alpha-5)}{10\alpha-5}},
\|\Lambda^{ \alpha}u \|^{\f{8\alpha}{ 6\alpha-5}}_{L^{2}} \|\Lambda^{2 \alpha}u \|_{L^{2}}^{2}\leq  CF_{0 ,1}^{\f{4\alpha }{6\alpha-5}}G_{0 ,1},$$
which enables us to  reformulate \eqref{3.50v1} as
\be\ba\label{3.50v2}
 \f{d}{dt}F_{1,1}+G_{1,1}
\leq C F_{1,1}^{\f{ (12\alpha-5)}{10\alpha-5}},
\ea
\ee
where we have chosen $\eta$ sufficiently small.\\
With the help of
\eqref{3.49} and Lemma \ref{hofflemma}, we deduce from \eqref{3.50v2} that
\be
G_{1,1} \in L^{\f{ (6\alpha-5)}{12\alpha-5}},
\ee
which yields
\be
\| \Lambda^{2\alpha} u_{t}  \|_{L^{2}}^{2},\| u_{tt}  \|_{L^{2}}^{2}  \in L^{\f{ (6\alpha-5)}{12\alpha-5}}.
\ee
Hence, we prove \eqref{3.40} for $r=1$.

Next, we assume that \eqref{3.40} is valid for $r\leq k-1$. It suffices to show  \eqref{3.40} for $r=k$. To this end, inductive hypothesis yields
\be\label{3.56}
\left\{\ba
&F_{j,0}\in L^\f{6\alpha-5}{4j\alpha+ 4\alpha-5},\\
&F_{j,1}\in L^\f{6\alpha-5}{4j\alpha+ 6\alpha-5},\\
&\| u^{( k) }_{t}\|_{L^{2}}^{2}\in L^{\f{6\alpha-5}{4k\alpha+4 \alpha-5}}. \ea\right.\ee
Thus,
\be\label{3.57} F_{k,0}=\|u_{t}^{(k)}\|_{L^{2}}^{2}+\sum_{j=1}^{k-1}\f{4j\alpha+4 \alpha-5}{4k\alpha+4 \alpha-5}F_{j,0}^{\f{4k\alpha+4 \alpha-5}{4j\alpha+4 \alpha-5}}+\sum_{j=0}^{k-1}\f{4j\alpha+6 \alpha-5}{4k\alpha+4 \alpha-5}F_{j,1}^{\f{4k\alpha+4 \alpha-5}{4j\alpha+6 \alpha-5}}\in L^{\f{6\alpha-5}{4k\alpha+4 \alpha-5}}.
 \ee
Using inductive hypothesis  \eqref{3.40} for $j\leq k-1$, we observe that
$$
\f{d}{dt}F_{j,1}+G_{j,1}\leq CF^{\f{4j\alpha+7 \alpha-5}{4j\alpha+6 \alpha-5}}_{j,1},\f{d}{dt}F_{j,0}+G_{j,0}\leq CF^{\f{4j\alpha+6 \alpha-5}{4j\alpha+4 \alpha-5}}_{r,0}.
$$
A  routine computation gives rise to
$$\ba
\f{4j\alpha+6 \alpha-5}{4k\alpha+4 \alpha-5}\f{d}{dt}F_{j,1}^{\f{4k\alpha+4 \alpha-5}{4j\alpha+6 \alpha-5}}+F_{j,1}^{\f{4(k-j)\alpha-2\alpha}{4j\alpha+6 \alpha-5}}G_{j,1}\leq CF^{\f{4k\alpha+6 \alpha-5}{4j\alpha+6 \alpha-5}}_{j,1},\\
\f{4j\alpha+4 \alpha-5}{4k\alpha+4 \alpha-5}\f{d}{dt}F_{j,0}^{\f{4k\alpha+4 \alpha-5}{4j\alpha+4 \alpha-5}}+F_{j,1}^{\f{4(k-j)\alpha }{4j\alpha+4 \alpha-5}}G_{j,0}\leq CF^{\f{4k\alpha+6 \alpha-5}{4j\alpha+4 \alpha-5}}_{j,0},
\ea$$
which together with \eqref{3.31v1} yields that
 \be\ba\label{3.57v1}
& \f{d}{dt}\B[\| u _{t}^{(k)}\|^{2}_{L^{2}}+\sum_{j=0}^{k-1}\f{4j\alpha+6 \alpha-5}{4k\alpha+4 \alpha-5}F_{j,1}^{\f{4k\alpha+4 \alpha-5}{4j\alpha+6 \alpha-5}}+\sum_{j=1}^{k-1}\f{4j\alpha+4 \alpha-5}{4k\alpha+4 \alpha-5}F_{j,0}^{\f{4k\alpha+4 \alpha-5}{4j\alpha+4 \alpha-5}}\B]\\&+\f12\|\Lambda^{ \alpha}u _{t}^{(k)}\|_{L^{2}}^{2} +\sum_{j=0}^{k-1}F_{j,1}^{\f{4(r-j)\alpha-2\alpha}{4j\alpha+6 \alpha-5}}G_{j,1} +\sum_{j=1}^{k-1}F_{j,1}^{\f{4(r-j)\alpha }{4j\alpha+4 \alpha-5}}G_{j,0}\\
\leq& C\|u_{t}^{(k)}\|
^{\f{2(4k\alpha+6\alpha-5)}{4k\alpha+4\alpha-5}}
_{L^{2}}
+C\sum_{j=1}^{k-1}\|u_{t}^{(j)}\|
^{\f{2(4k\alpha+6\alpha-5)}{4j\alpha+4\alpha-5}}
_{L^{2}}
+C\sum_{j=0}^{k-1}\|\Lambda^{\alpha}u_{t}^{( j)} \|_{L^{2}}^{\f{2(4k\alpha+6\alpha-5)}{4j\alpha+6\alpha-5}}
\\&+2\eta\sum_{j=1}^{k-1}\|\Lambda^{\alpha}u_{t}^{( j)} \|_{L^{2}}^{\f{4\alpha(2(k-j)-1)}{  4j\alpha+6\alpha-5}}
\|\Lambda^{2\alpha}u_{t}^{(j)} \|^{2}_{L^{2}}+C\sum_{j=0}^{k-1}F^{\f{4k\alpha+6 \alpha-5}{4j\alpha+6 \alpha-5}}_{j,1}+C\sum_{j=1}^{k-1}F^{\f{4k\alpha+6 \alpha-5}{4j\alpha+4 \alpha-5}}_{j,0}.
\ea
\ee
Making use of definition of $F_{j,0}$ and  $F_{j,1}$ in \eqref{3.41}, we  have
$\|\Lambda^{\alpha}u_{t}^{( j)} \|_{L^{2}} \leq CF_{j,1}$ and
$\|  u_{t}^{( j)} \|_{L^{2}} \leq CF_{j,0}.$  In addition, there holds
$$\|\Lambda^{\alpha}u_{t}^{( j)} \|_{L^{2}}^{\f{4\alpha(2(k-j)-1)}{  4j\alpha+6\alpha-5}}
\|\Lambda^{2\alpha}u_{t}^{(j)} \|^{2}_{L^{2}}\leq CF_{j,1}^{\f{4(k-j)\alpha-2\alpha}{4j\alpha+6 \alpha-5}}G_{j,1}.$$
As a consequence, we  take $\eta$ sufficient small to rewrite \eqref{3.57v1}
as
$$\ba
& \f{d}{dt}\B[\| u _{t}^{(k)}\|^{2}_{L^{2}}+\sum_{j=0}^{k-1}\f{4j\alpha+6 \alpha-5}{4k\alpha+4 \alpha-5}F_{j,1}^{\f{4k\alpha+4 \alpha-5}{4j\alpha+6 \alpha-5}}+\sum_{j=1}^{k-1}\f{4j\alpha+4 \alpha-5}{4k\alpha+4 \alpha-5}F_{j,0}^{\f{4k\alpha+4 \alpha-5}{4j\alpha+4 \alpha-5}}\B]\\&+\f12\|\Lambda^{ \alpha}u _{t}^{(k)}\|_{L^{2}}^{2} +\f12\sum_{j=0}^{k-1}F_{j,1}^{\f{4(r-j)\alpha-2\alpha}{4j\alpha+6 \alpha-5}}G_{j,1} +\f12\sum_{j=1}^{k-1}F_{j,1}^{\f{4(r-j)\alpha }{4j\alpha+4 \alpha-5}}G_{j,0}\\
\leq& C\|u_{t}^{(k)}\|
^{\f{2(4k\alpha+6\alpha-5)}{4k\alpha+4\alpha-5}}
_{L^{2}}
  +C\sum_{j=0}^{k-1}F^{\f{4k\alpha+6 \alpha-5}{4j\alpha+6 \alpha-5}}_{j,1}+C\sum_{j=1}^{k-1}F^{\f{4k\alpha+6 \alpha-5}{4j\alpha+4 \alpha-5}}_{j,0},
\ea
$$
Thanks to the definition of \eqref{3.41}, we see that
$$
C\sum_{j=0}^{k-1}F^{\f{4k +1}{4j +1}}_{j,1}\leq CF_{k,0}^{\f{ 4k\alpha+6\alpha-5 }{4k\alpha+4\alpha-5}},~~
C\sum_{j=1}^{k-1}F^{\f{4k +1}{4j -1}}_{j,0}\leq CF_{k,0}^{\f{ 4k\alpha+6\alpha-5 }{4k\alpha+4\alpha-5}}.
 $$
 Hence, we find
\be\label{3.58}
\f{d}{dt}F_{k,0}+G_{k,0}\leq CF_{k,0}^{\f{ 4k\alpha+6\alpha-5 }{4k\alpha+4\alpha-5}}.
\ee
Applying Lemma \ref{hofflemma}, we conclude by \eqref{3.58}  and \eqref{3.57} that
$$
G_{k,0}\in L^{\f{6\alpha-5}{4k\alpha+6\alpha  -5}},
$$
which turns out that
\be\label{3.59}
\|\Lambda^{ \alpha}u _{t}^{(k)}\|_{L^{2}}^{2} \in L^{\f{6\alpha-5}{4k\alpha+6\alpha  -5}}.\ee
This shows $\eqref{3.40}_{1}$-$\eqref{3.40}_{3}$ for $r=k$. We turn our attentions to demonstrating  $\eqref{3.40}_{4}$-$\eqref{3.40}_{6}$ for $r=k$.
Combining \eqref{3.59} and \eqref{3.56}, we discover that
\be\label{3.62}
F_{k,1}=\|\Lambda^{\alpha}u_{t}^{(k)}\|_{L^{2}}^{2}+\sum_{j=0}^{k-1}\f{4j\alpha+6 \alpha-5}{4k\alpha+6 \alpha-5}F_{j,1}^{\f{4k\alpha+6 \alpha-5}{4j\alpha+6 \alpha-5}}\in L^{\f{6\alpha-5}{4k\alpha+6\alpha  -5}}.\ee
Since
\eqref{3.40} is valid for $r\leq k-1$, we find
$$
\f{d}{dt}F_{j,1}+G_{j,1}\leq CF^{\f{4j\alpha+7 \alpha-5}{4j\alpha+6 \alpha-5}}_{j,1},\f{d}{dt}F_{r,0}+G_{r,0}\leq CF^{\f{4r\alpha+6 \alpha-5}{4r\alpha+4 \alpha-5}}_{r,0},
$$
which means
\be\ba\label{3.63}
\f{4j\alpha+6 \alpha-5}{4k\alpha+6 \alpha-5}\f{d}{dt}F_{j,1}^{\f{4k\alpha+6 \alpha-5}{4j\alpha+6 \alpha-5}}+F_{j,1}^{\f{4(k-j)\alpha }{4j\alpha+6 \alpha-5}}G_{j,1}\leq CF^{\f{4k\alpha+8 \alpha-5}{4j\alpha+6 \alpha-5}}_{j,1}.
\ea\ee
Taking advantaging of definition of \eqref{3.41}, one has
$$
\sum_{j>0}\|\Lambda^{\alpha}u_{t}^{(r-j)} \|_{L^{2}}^{\f{8\alpha j}{  4(r-j)\alpha+6\alpha-5}}
\|\Lambda^{2\alpha}u_{t}^{(r-j)} \|^{2}_{L^{2}}\leq C\sum_{j=0}^{k-1}F_{j,1}^{\f{4(k-j)\alpha }{4j\alpha+6 \alpha-5}}G_{j,1}.$$
As a consequence, we derive from \eqref{3.38v1} and \eqref{3.63}  that
\be\ba\label{3.64}
& \f{d}{dt}\B[\| \Lambda^{\alpha}u _{t}^{(k)}\|^{2}_{L^{2}}+\sum_{j=0}^{k-1}\f{4j\alpha+6 \alpha-5}{4k\alpha+6 \alpha-5}F_{j,1}^{\f{4k\alpha+6 \alpha-5}{4j\alpha+6 \alpha-5}}\B]\\&+\f12\|\Lambda^{2 \alpha}u _{t}^{(k)}\|_{L^{2}}^{2} +\f12\| u _{t}^{(k+1)}\|_{L^{2}}^{2}+\f12\sum_{j=0}^{k-1}F_{j,1}^{\f{4(k-j)\alpha }{4j\alpha+6 \alpha-5}}G_{j,1}\\ \leq& C\|\Lambda^{\alpha}u_{t}^{(k )}  \|_{L^{2}}^{\f{2(4k\alpha+8\alpha-5)}{ 4k\alpha +6\alpha-5}}+C\sum_{j=0}^{k-1}F_{j,1}^{ \f{4 k\alpha+8\alpha-5}{4j\alpha+6\alpha-5}},
\ea
\ee
where we have chosen $\eta$  sufficiently small.\\
In view of definition of \eqref{3.41}, we get
$$ \sum_{j=0}^{k-1}F_{j,1}^{ \f{4 k\alpha+8\alpha-5}{4j\alpha+6\alpha-5}}\leq CF_{k,1}^{\f{ (4k\alpha+8\alpha-5)}{4k\alpha+6\alpha-5}},
$$
which together with \eqref{3.64} yields that
$$
\f{d}{dt}F_{k,1}+G_{k,1}\leq CF_{k,1}^{\f{ (4k\alpha+8\alpha-5)}{4k\alpha+6\alpha-5}}.
$$
In view of $
F_{k,1} \in L^{\f{6\alpha-5}{4k\alpha+6\alpha  -5}} $ in \eqref{3.62} and Lemma \ref{hofflemma}, we derive from the latter inequality that  $
G_{k,1} \in L^{\f{6\alpha-5}{4k\alpha+8\alpha  -5}}. $ We end up with $\| u _{t}^{(k+1)}\|_{L^{2}}^{2}\in L^{\f{6\alpha-5}{4k\alpha+8\alpha  -5}}$. We prove $\eqref{3.40}$ for $r=k$.
This completes the proof of the theorem.
\end{proof}
 \section{Space-time derivative estimates of solutions to the fractional Navier-Stokes equations}

The goal of this section is to show that space-time derivative estimates $\Lambda^{n\alpha}u^{(m)}_{t}  \in L^{\f{2(6\alpha-5)}{4m\alpha+2n\alpha+4 \alpha-5}}L^{2}$ for any $m, n \in \mathbb{N}$.   We will apply time derivative estimates \eqref{timederivative} in the deduction of regularity  \eqref{1.1} of solutions of the fractional Navier-Stokes equations.

\begin{lemma}Assume that $u$ is a smooth solution of   the   fractional Navier-Stokes equations \eqref{GNS}. Then for each $t >0$ and $\eta >0$, there hold
\be\ba\label{4.13v2}
&\f{d}{dt} \| \Lambda^{s\alpha}u _{t}^{(r)}\|^{2}_{L^{2}}+\f{1}{2}\| \Lambda^{(s+1)\alpha}u _{t}^{(r)}\|^{2}_{L^{2}}\\\leq&C\sum_{j }^{r-1} \|\Lambda^{ s \alpha}u_{t}^{(j)}\|_{L^{2}}
^{\f{2(4r\alpha+2(s+1)\alpha+4\alpha-5)}
{4j\alpha+2s\alpha+4\alpha-5}}+C\|\Lambda^{ s \alpha}u_{t}^{(r)}\|_{L^{2}}
^{\f{2(4r\alpha+2(s+1)\alpha+4\alpha-5)}
{4r\alpha+2s\alpha+4\alpha-5}}\\&+C\sum_{j }^{r}
\|\Lambda^{\alpha}u^{(r-j)}_{t} \|^{\f{2(4r\alpha+2(s+1)\alpha+4\alpha-5)}{4(r-j)\alpha+2 \alpha+4\alpha-5}}_{L^{2}}\\&+2\eta\sum_{j>0 }^{r}\|\Lambda^{s\alpha}u^{(r-j )}_{t}\|^{\f{8\alpha j}{ 4(r-j)\alpha+2 s\alpha+4\alpha-5}}_{L^{2}}\|\Lambda^{(s+1)\alpha} u^{(r-j )}_{t}\|_{L^{2}}^{2},
\ea\ee
and
\be\ba\label{4.15v1}
 \f{1}{2}\f{d}{dt} \| \Lambda^{s\alpha}u \|^{2}_{L^{2}}+\| \Lambda^{(s+1)\alpha}u \|^{2}_{L^{2}} \leq C\|\Lambda^{ s \alpha}u \|_{L^{2}}
^{\f{2( 2(s+1)\alpha+4\alpha-5)}
{ 2s\alpha+4\alpha-5}} +C
\|\Lambda^{\alpha}u \|^{\f{2( 2(s+1)\alpha+4\alpha-5)}{  2 \alpha+4\alpha-5}}_{L^{2}}.
\ea
\ee
\end{lemma}
\begin{proof}
Leibniz rule  enables us to rewrite the   fractional Navier-Stokes
equations \eqref{GNS} as
\be\ba\label{4.1} \Lambda^{s\alpha} u^{(r+1) }_{t}+ (-\Delta)^{\alpha}\Lambda^{s\alpha}u^{(r)}_{t}+\sum_{j=0}^{r}\binom{r}{j}\Lambda^{s\alpha}(u_{t}^{(j)}\cdot\nabla u^{(r-j)}_{t}) +\nabla \Lambda^{s\alpha}p^{(r)}_{t}=0.
\ea
\ee
Taking the $L^{2}$ inner product of the above equation with $\Lambda^{s\alpha}u _{t}^{(r)}$, we obtain
\be\ba\label{4.2}
&\f{1}{2}\f{d}{dt} \| \Lambda^{s\alpha}u _{t}^{(r)}\|^{2}_{L^{2}}+\| \Lambda^{(s+1)\alpha}u _{t}^{(r)}\|^{2}_{L^{2}}=- I,
\ea
\ee
where
\be\ba
I=\langle\Lambda^{s\alpha}(u \cdot\nabla u^{(r )}_{t}), \Lambda^{s\alpha}u _{t}^{(r)}\rangle+\sum_{j>0}\binom{r}{j}\langle \Lambda^{s\alpha}(u_{t}^{(j)}\cdot\nabla u^{(r-j)}_{t}), \Lambda^{s\alpha}u _{t}^{(r)}\rangle.
\ea
\ee
Taking advantage of incompressiable condition, we have
\be\ba \langle\Lambda^{s\alpha}(u \cdot\nabla u^{(r )}_{t}), \Lambda^{s\alpha}u _{t}^{(r)}\rangle=\langle[\Lambda^{s\alpha}(u \cdot\nabla u^{(r )}_{t})-u \cdot\Lambda^{s\alpha}\nabla u^{(r )}_{t})], \Lambda^{s\alpha}u _{t}^{(r)}\rangle.
\ea
\ee
The Kato-Ponce type     commutator estimate \eqref{katoponce} allows us to get
\be\ba\label{3.5} &\B|\langle\Lambda^{s\alpha}(u \cdot\nabla u^{(r )}_{t})\cdot\Lambda^{s\alpha}u _{t}^{(r)}\rangle\B|\\
\leq& (\|\Lambda^{s\alpha}u\|_{L^{2}}\|\nabla u^{(r )}_{t}\|_{L^{\f{3}{\alpha}}}+\|\nabla u\|_{L^{\f{3}{\alpha}}}\|\Lambda^{s\alpha} u^{(r )}_{t}\|_{L^{2}})\|\Lambda^{s\alpha}u _{t}^{(r)}\|_{L^{\f{6}{3-2\alpha}}}.
\ea
\ee
Applying interpolation inequality \eqref{keyinequality},
we get
$$\ba
&\|\nabla u^{(r )}_{t}\|_{L^{\f{3}{\alpha}}}\leq C\|\Lambda^{\alpha}u^{(r )}_{t}\|^{\f{2(s+1)\alpha -5 }{2\alpha(s-1)}}_{L^{2}}\|\Lambda^{s\alpha} u^{(r )}_{t}\|_{L^{2}}^{\f{5-4\alpha}{2\alpha(s-1)}},\\
&\|\nabla u \|_{L^{\f{3}{\alpha}}}\leq C\|\Lambda^{\alpha}u \|^{\f{2(s+1)\alpha -5 }{2\alpha(s-1)}}_{L^{2}}\|\Lambda^{s\alpha} u \|_{L^{2}}^{\f{5-4\alpha}{2\alpha(s-1)}}.
\ea
$$
Substituting the previous inequalities into \eqref{3.5}, we conclude by Cauchy-Schwarz inequality  that
\be\ba\label{4.7} &\B| \langle\Lambda^{s\alpha}(u \cdot\nabla u^{(r )}_{t})\cdot\Lambda^{s\alpha}u _{t}^{(r)}\rangle\B|\\
\leq& C\|\Lambda^{s\alpha}u\|_{L^{2}}\|\Lambda^{\alpha}u^{(r )}_{t}\|^{\f{2(s+1)\alpha -5 }{2\alpha(s-1)}}_{L^{2}}\|\Lambda^{s\alpha} u^{(r )}_{t}\|_{L^{2}}^{\f{5-4\alpha}{2\alpha(s-1)}}\|\Lambda^{(s+1)\alpha}u _{t}^{(r)}\|_{L^{2}}\\&+C\|\Lambda^{\alpha}u \|^{\f{2(s+1)\alpha -5 }{2\alpha(s-1)}}_{L^{2}}\|\Lambda^{s\alpha} u \|_{L^{2}}^{\f{5-4\alpha}{2\alpha(s-1)}}\|\Lambda^{s\alpha} u^{(r )}_{t}\|_{L^{2}} \|\Lambda^{(s+1)\alpha}u _{t}^{(r)}\|_{L^{2}}\\
\leq& C\|\Lambda^{s\alpha}u\|_{L^{2}}^{2}\|\Lambda^{\alpha}u^{(r )}_{t}\|^{\f{2(s+1)\alpha -5 }{ \alpha(s-1)}}_{L^{2}}\|\Lambda^{s\alpha} u^{(r )}_{t}\|_{L^{2}}^{\f{5-4\alpha}{ \alpha(s-1)}} \\&+C\|\Lambda^{s\alpha} u^{(r )}_{t}\|^{2}_{L^{2}}\|\Lambda^{\alpha}u \|^{\f{2(s+1)\alpha -5 }{ \alpha(s-1)}}_{L^{2}}\|\Lambda^{s\alpha} u \|_{L^{2}}^{\f{5-4\alpha}{ \alpha(s-1)}} +\f\varepsilon2\|\Lambda^{(s+1)\alpha}u _{t}^{(r)}\|^{2}_{L^{2}},
\ea
\ee
where $\varepsilon$ will be determined later.\\
In view of the H\"older inequality, Kato-Ponce type   Leibniz rules \eqref{fraleibnilaw} and the Cauchy-Schwarz inequality, we observe that
\be\ba\label{4.8}
 &\B|\sum_{j>0}\binom{r}{j}\langle\Lambda^{s\alpha}(u_{t}^{(j)}\cdot\nabla u^{(r-j)}_{t}),\Lambda^{s\alpha}u _{t}^{(r)}\rangle\B|\\
  = & \B|\binom{r}{j}\sum_{j>0}\langle\Lambda^{(s-1)\alpha}(u_{t}^{(j)}\cdot\nabla u^{(r-j)}_{t})\cdot\Lambda^{(s+1)\alpha}u _{t}^{(r)}\rangle\B|\\
 \leq & C \sum_{j>0}\|\Lambda^{(s-1)\alpha}u_{t}^{(j)}\|_{L^{\f{6}{3-2\alpha}}}\|\nabla u^{(r-j)}_{t}\|_{L^{\f{3}{\alpha}}}\|\Lambda^{(s+1)\alpha}u _{t}^{(r)}\|_{L^{2}}
\\& +C\sum_{j>0}\|u_{t}^{(j)}\|_{L^{\f{6}{3-2\alpha}}}\|\Lambda^{(s-1)\alpha}\nabla u^{(r-j)}_{t}\|_{L^{\f{3}{\alpha}}}\|\Lambda^{(s+1)\alpha}u _{t}^{(r)}\|_{L^{2}}
\\
 \leq & C\sum_{j>0}\|\Lambda^{ s \alpha}u_{t}^{(j)}\|_{L^{2}}\|\Lambda^{\alpha}u^{(r-j)}_{t} \|^{\f{2(s+1)\alpha -5 }{2\alpha(s-1)}}_{L^{2}}\|\Lambda^{s\alpha} u^{(r-j)}_{t} \|_{L^{2}}^{\f{5-4\alpha}{2\alpha(s-1)}}\|\Lambda^{(s+1)\alpha}u _{t}^{(r)}\|_{L^{2}}
 \\&+C\sum_{j>0}\|\Lambda^{\alpha } u_{t}^{(j)}\|_{L^{2}}\|\Lambda^{s\alpha}u^{(r-j )}_{t}\|^{\f{2(s+1)\alpha -5 }{2\alpha(s-1)}}_{L^{2}}\|\Lambda^{(s+1)\alpha} u^{(r-j )}_{t}\|_{L^{2}}^{\f{5-4\alpha}{2\alpha(s-1)}} \|\Lambda^{(s+1)\alpha}u _{t}^{(r)}\|_{L^{2}}\\
 \leq & C\sum_{j>0}\|\Lambda^{ s \alpha}u_{t}^{(j)}\|_{L^{2}}^{2}\|\Lambda^{\alpha}u^{(r-j)}_{t} \|^{\f{2(s+1)\alpha -5 }{ \alpha(s-1)}}_{L^{2}}\|\Lambda^{s\alpha} u^{(r-j)}_{t} \|_{L^{2}}^{\f{5-4\alpha}{ \alpha(s-1)}}
 \\&+C\sum_{j>0}\|\Lambda^{\alpha } u_{t}^{(j)}\|^{2}_{L^{2}}\|\Lambda^{s\alpha}u^{(r-j )}_{t}\|^{\f{2(s+1)\alpha -5 }{ \alpha(s-1)}}_{L^{2}}\|\Lambda^{(s+1)\alpha} u^{(r-j )}_{t}\|_{L^{2}}^{\f{5-4\alpha}{ \alpha(s-1)}} +\f\varepsilon2\|\Lambda^{(s+1)\alpha}u _{t}^{(r)}\|_{L^{2}}^{2}.
 \ea\ee
 From \eqref{4.7}-\eqref{4.8}, we observe that
 \be\ba
&|-\sum_{j}^{r}\binom{r}{j}\langle\Lambda^{s\alpha}(u_{t}^{(j)}\cdot\nabla u^{(r-j)}_{t})\cdot\Lambda^{s\alpha}u _{t}^{(r)}\rangle|
\\
\leq& C\sum_{j }^{r}\|\Lambda^{ s \alpha}u_{t}^{(j)}\|_{L^{2}}^{2}\|\Lambda^{\alpha}u^{(r-j)}_{t} \|^{\f{2(s+1)\alpha -5 }{ \alpha(s-1)}}_{L^{2}}\|\Lambda^{s\alpha} u^{(r-j)}_{t} \|_{L^{2}}^{\f{5-4\alpha}{ \alpha(s-1)}}
 \\&+C\sum_{j>0 }^{r}\|\Lambda^{\alpha } u_{t}^{(j)}\|^{2}_{L^{2}}\|\Lambda^{s\alpha}u^{(r-j )}_{t}\|^{\f{2(s+1)\alpha -5 }{ \alpha(s-1)}}_{L^{2}}\|\Lambda^{(s+1)\alpha} u^{(r-j )}_{t}\|_{L^{2}}^{\f{5-4\alpha}{ \alpha(s-1)}} +\varepsilon\|\Lambda^{(s+1)\alpha}u _{t}^{(r)}\|_{L^{2}}^{2}.
\ea\ee
Inserting this into \eqref{4.2} and taking $\varepsilon$ small, we notice that
\be\ba
&\f{d}{dt} \| \Lambda^{s\alpha}u _{t}^{(r)}\|^{2}_{L^{2}}+\f{1}{2}\| \Lambda^{(s+1)\alpha}u _{t}^{(r)}\|^{2}_{L^{2}}\\
\leq& C\sum_{j }^{r}\|\Lambda^{ s \alpha}u_{t}^{(j)}\|_{L^{2}}^{2}\|\Lambda^{\alpha}u^{(r-j)}_{t} \|^{\f{2(s+1)\alpha -5 }{ \alpha(s-1)}}_{L^{2}}\|\Lambda^{s\alpha} u^{(r-j)}_{t} \|_{L^{2}}^{\f{5-4\alpha}{ \alpha(s-1)}}
 \\&+C\sum_{j>0 }^{r}\|\Lambda^{\alpha } u_{t}^{(j)}\|^{2}_{L^{2}}\|\Lambda^{s\alpha}u^{(r-j )}_{t}\|^{\f{2(s+1)\alpha -5 }{ \alpha(s-1)}}_{L^{2}}\|\Lambda^{(s+1)\alpha} u^{(r-j )}_{t}\|_{L^{2}}^{\f{5-4\alpha}{ \alpha(s-1)}}.
\ea\ee
The Young inequality ensures that
\be\ba
&C\sum_{j }^{r}\|\Lambda^{ s \alpha}u_{t}^{(j)}\|_{L^{2}}^{2}\|\Lambda^{\alpha}u^{(r-j)}_{t} \|^{\f{2(s+1)\alpha -5 }{ \alpha(s-1)}}_{L^{2}}\|\Lambda^{s\alpha} u^{(r-j)}_{t} \|_{L^{2}}^{\f{5-4\alpha}{ \alpha(s-1)}}\\
\leq& C\sum_{j }^{r}\|\Lambda^{ s \alpha}u_{t}^{(j)}\|_{L^{2}}
^{\f{2(4r\alpha+2(s+1)\alpha+4\alpha-5)}{4j\alpha+2s\alpha+4\alpha-5}}+C
\sum_{j }^{r}\|\Lambda^{\alpha}u^{(r-j)}_{t} \|^{\f{2(4r\alpha+2(s+1)\alpha+4\alpha-5)}{4(r-j)\alpha+2 \alpha+4\alpha-5}}_{L^{2}}\\&+C\sum_{j }^{r}
\|\Lambda^{s\alpha}u^{(r-j)}_{t} \|^{\f{2(4r\alpha+2(s+1)\alpha+4\alpha-5)}{4(r-j)\alpha+2 s\alpha+4\alpha-5}}_{L^{2}},
\ea\ee
and
\be\ba
&C\sum_{j>0 }^{r}\|\Lambda^{\alpha } u_{t}^{(j)}\|^{2}_{L^{2}}\|\Lambda^{s\alpha}u^{(r-j )}_{t}\|^{\f{2(s+1)\alpha -5 }{ \alpha(s-1)}}_{L^{2}}\|\Lambda^{(s+1)\alpha} u^{(r-j )}_{t}\|_{L^{2}
}^{\f{5-4\alpha}{ \alpha(s-1)}}\\
\leq&C\sum_{j>0 }^{r}
\|\Lambda^{\alpha } u_{t}^{(j)}\|^{\f{2(4r\alpha+2(s+1)\alpha+4\alpha-5)}{4j\alpha +6\alpha-5}}_{L^{2}}
+C\sum_{j>0 }^{r}\|\Lambda^{s\alpha}u^{(r-j )}_{t}\|^{\f{2(4r\alpha+2(s+1)\alpha+4\alpha-5)}{ 4(r-j)\alpha+2 s\alpha+4\alpha-5}}_{L^{2}}\\&+\eta\sum_{j>0 }^{r}\|\Lambda^{s\alpha}u^{(r-j )}_{t}\|^{\f{8\alpha j}{ 4(r-j)\alpha+2 s\alpha+4\alpha-5}}_{L^{2}}\|\Lambda^{(s+1)\alpha} u^{(r-j )}_{t}\|_{L^{2}}^{2},
\ea\ee
which helps us to obtain
\be\ba\label{4.13v1}
&\f{d}{dt} \| \Lambda^{s\alpha}u _{t}^{(r)}\|^{2}_{L^{2}}+\f{1}{2}\| \Lambda^{(s+1)\alpha}u _{t}^{(r)}\|^{2}_{L^{2}}\\\leq&C\sum_{j }^{r-1} \|\Lambda^{ s \alpha}u_{t}^{(j)}\|_{L^{2}}
^{\f{2(4r\alpha+2(s+1)\alpha+4\alpha-5)}
{4j\alpha+2s\alpha+4\alpha-5}}+C\|\Lambda^{ s \alpha}u_{t}^{(r)}\|_{L^{2}}
^{\f{2(4r\alpha+2(s+1)\alpha+4\alpha-5)}
{4r\alpha+2s\alpha+4\alpha-5}}\\&+C\sum_{j }^{r}
\|\Lambda^{\alpha}u^{(r-j)}_{t} \|^{\f{2(4r\alpha+2(s+1)\alpha+4\alpha-5)}{4(r-j)\alpha+2 \alpha+4\alpha-5}}_{L^{2}}\\&+ \eta\sum_{j>0 }^{r}\|\Lambda^{s\alpha}u^{(r-j )}_{t}\|^{\f{8\alpha j}{ 4(r-j)\alpha+2 s\alpha+4\alpha-5}}_{L^{2}}\|\Lambda^{(s+1)\alpha} u^{(r-j )}_{t}\|_{L^{2}}^{2}.
\ea\ee
Here, $\eta$ will be fixed later.\\
It remains to show \eqref{4.15v1}.
Multiply equality  \eqref{4.1} with $r=0$  by $\Lambda^{s\alpha}u $ to discover
\be\ba\label{4.13v3}
&\f{1}{2}\f{d}{dt} \| \Lambda^{s\alpha}u \|^{2}_{L^{2}}+\| \Lambda^{(s+1)\alpha}u \|^{2}_{L^{2}}=-\sum_{j}\langle\Lambda^{s\alpha}(u \cdot\nabla u )\cdot\Lambda^{s\alpha}u \rangle.
\ea
\ee
A slight modification the proof of \eqref{4.7}, we have
\be\ba  &\B|\int_{\mathbb{R}^{3}}\Lambda^{s\alpha}(u \cdot\nabla u )\cdot\Lambda^{s\alpha}u dx\B|\\
\leq& C \|\Lambda^{\alpha}u \|^{\f{2(s+1)\alpha -5 }{ \alpha(s-1)}}_{L^{2}}\|\Lambda^{s\alpha} u \|_{L^{2}}^{2+\f{5-4\alpha}{ \alpha(s-1)}}   +C\|\Lambda^{(s+1)\alpha}u \|^{2}_{L^{2}}\\\leq & C\|\Lambda^{ s \alpha}u \|_{L^{2}}
^{\f{2( 2(s+1)\alpha+4\alpha-5)}
{ 2s\alpha+4\alpha-5}} +C
\|\Lambda^{\alpha}u \|^{\f{2( 2(s+1)\alpha+4\alpha-5)}{  2 \alpha+4\alpha-5}}_{L^{2}} +\varepsilon\|\Lambda^{(s+1)\alpha}u \|^{2}_{L^{2}}.
\ea
\ee
Plugging this into \eqref{4.13}, we arrive at
\be\ba\label{4.15}
 \f{1}{2}\f{d}{dt} \| \Lambda^{s\alpha}u \|^{2}_{L^{2}}+\| \Lambda^{(s+1)\alpha}u \|^{2}_{L^{2}} \leq C\|\Lambda^{ s \alpha}u \|_{L^{2}}
^{\f{2( 2(s+1)\alpha+4\alpha-5)}
{ 2s\alpha+4\alpha-5}} +C
\|\Lambda^{\alpha}u \|^{\f{2( 2(s+1)\alpha+4\alpha-5)}{  2 \alpha+4\alpha-5}}_{L^{2}}.
\ea
\ee
This completes the proof of the
lemma.
\end{proof}
\newpage
To prove Theorem \ref{the1.1}, we introduce
\be\ba
&F_{p,q}= \|\Lambda^{q\alpha}u_{t}^{(p)}\|_{L^{2}}^{2}+\sum_{\ell=0}^{p-1}
\f{4\ell\alpha+2q\alpha+4 \alpha-5}{4p\alpha+2q\alpha+4 \alpha-5}F_{\ell,q}^{}+\sum_{\ell=0}^{p }
\f{4\ell\alpha+6 \alpha-5}{4p\alpha+2q\alpha+4 \alpha-5}F_{\ell,1}^{\f{4p\alpha+2q\alpha+4 \alpha-5}{4\ell\alpha+6 \alpha-5}},\\
&F_{0,q}= \|\Lambda^{q\alpha}u \|_{L^{2}}^{2}+\f{6 \alpha-5}{ 2q\alpha+4 \alpha-5}
F_{0,1}^{\f{ 2q\alpha+4 \alpha-5}{6 \alpha-5}},\\
&
G_{p,q}= \f12\|\Lambda^{(q+1)\alpha}u_{t}^{(p)}\|_{L^{2}}^{2}+\f12\sum_{\ell=0}^{p-1}
F_{\ell,q}^{\f{4(p-\ell)\alpha }{4\ell\alpha+2q\alpha+4 \alpha-5}}G_{\ell,q},\\
&G_{0,q}= \f12\|\Lambda^{(q+1)\alpha}u \|_{L^{2}}^{2}+\f12
F_{0,1}^{\f{2q\alpha-2\alpha }{6 \alpha-5}}G_{0,1},\\
&G_{p,0}= \f12\|\Lambda^{ \alpha}u_{t}^{(p)}\|_{L^{2}}^{2}+\f12\sum_{\ell=0}^{p-1}
F_{\ell,1}^{\f{4(p-\ell)\alpha -2\alpha}{4\ell\alpha +6 \alpha-5}}G_{\ell,1}.
\ea\ee
The above definitions of notations of $F_{p,q}$, $G_{p,q}, $ $G_{0,q} $ and $ G_{p,0}$ rest on the energy inequalities  \eqref{4.13v2} and \eqref{4.15v1}.

\begin{proof}[Proof of Theorem \ref{the1.1}]
It suffices to show that
\be\label{4.13}
\left\{\ba
&F_{p,q}\in L^{\f{6\alpha-5}{4p\alpha+2q\alpha+4\alpha-5}},\\
&\f{d}{dt}F_{p,q}+G_{p,q}\leq C F^{\f{4p\alpha+2q\alpha+6 \alpha-5}{4p\alpha+2q\alpha+4 \alpha-5}}_{p,q},\\
&G_{p,q}\in L^{\f{6\alpha-5}{4p\alpha+2q\alpha+6 \alpha-5}}, \|\Lambda^{( q+1)\alpha}u^{(p ) }_{t}\|_{L^{2}}^{2}\in L^{\f{6\alpha-5}{4p\alpha+2q\alpha+6 \alpha-5}}.\ea\right.\ee
 We will apply inductive method to demonstrate  \eqref{4.13} for all $p, q\in \mathbb{N}$.
 To make the paper more readable, we denote
\be\ba\label{4.16}
&F_{0,1}= \|\Lambda^{ \alpha}u \|_{L^{2}}^{2}, \\
&F_{1,0}=\| u_{t}  \|_{L^{2}}^{2}+\f{6\alpha-5}{ 8\alpha-5 }\|\Lambda^{\alpha}u  \|_{L^{2}}^{\f{2(8\alpha-5)}{6\alpha-5}}=\| u_{t}  \|_{L^{2}}^{2}+\f{6\alpha-5}{ 8\alpha-5 }F_{0,1}^{\f{ 8\alpha-5}{6\alpha-5}},\\
&F_{0,2}= \|\Lambda^{2\alpha}u \|_{L^{2}}^{2}+\f{6\alpha-5}{8\alpha-5} F_{0,1}^{\f{8\alpha-5}{6\alpha-5}},\\
&F_{1,1}=  \|\Lambda^{ \alpha}u_{t} \|_{L^{2}}^{2}+\f{ 6 \alpha-5}{ 10\alpha-5}
F_{0,1}^{\f{ 10\alpha-5}{ 6 \alpha-5}}, \\
&F_{2,0}= \| u_{t}^{(2)}\|_{L^{2}}^{2}+
\f{8 \alpha-5}{ 12 \alpha-5}F_{1,0}^{\f{ 12 \alpha-5}{8 \alpha-5}}+\f{6 \alpha-5}{ 12 \alpha-5}
F_{0,1}^{\f{ 12 \alpha-5}{6 \alpha-5}}+\f{10 \alpha-5}{ 12 \alpha-5}F_{1,1}^{\f{ 12 \alpha-5}{10 \alpha-5}},\\
&G_{0,1}= \f12\|\Lambda^{2\alpha}u \|_{L^{2}}^{2}, \\
&G_{1,1}= \f12\|\Lambda^{ 2 \alpha}u_{t} \|_{L^{2}}^{2}+\f12
F_{0,1}^{\f{4 \alpha }{6 \alpha-5}}G_{0,1},\\
&G_{0,2}= \f12\|\Lambda^{3\alpha}u \|_{L^{2}}^{2} +\f12F^{\f{2\alpha}{6\alpha-5}}_{0,1}G_{0,1},\\
&G_{2,0}= \f12\|\Lambda^{\alpha} u_{t}^{(2)}\|_{L^{2}}^{2}+\f12
F_{0,1}^{\f{ 6 \alpha }{6 \alpha-5}}G_{0,1}+\f12F_{1,1}^{\f{2 \alpha }{10 \alpha-5}}G_{1,1}.
\ea\ee
In last section, we have proved    \eqref{4.13} for $p+q=1$.
Indeed, we derive from \eqref{3.42}-\eqref{3.48}
that
\be\label{4.21v1}
\left\{\ba
&F_{0,1} \in L^1,F_{ 1,0} \in   L^{\f{ 6\alpha-5 }{ 8\alpha-5}},\\
&\f{d}{dt}F_{0,1}+G_{0,1}\leq CF_{0,1}^{\f{ 8\alpha-5}{6\alpha-5}},\f{d}{dt}F_{ 1,0}+G_{1,0}\leq CF_{ 1,0}^{\f{ 10\alpha-5}{8\alpha-5}},\\
&G_{0,1}\in L^{\f{6\alpha-5}{ 8\alpha-5}}
,  \| \Lambda^{2\alpha}u   \|_{L^{2}}^{2}\in L^{\f{6\alpha-5}{ 8\alpha-5}},G_{1,0}\in L^{\f{6\alpha-5}{10 \alpha-5}},\|\Lambda^{\alpha}u  _{t}\|_{L^{2}}^{2}\in L^{\f{6\alpha-5}{10\alpha-5}},\ea\right.\ee
which yields \eqref{4.13} for $p+q=1$.

 The derivation of   \eqref{4.13}      for $   p+q=k\geq2$
 becomes   quite involved.
 Next we shall consider the case $p+q=2$ for a heuristic analysis. We deduce from
\eqref{4.16}-\eqref{4.21v1} that $F_{0,2} \in L^{\f{ 6 \alpha-5}{8 \alpha-5}}$.\\
On account of \eqref{4.15}, we see that
\be\ba\label{4.21}
&\f{d}{dt} \| \Lambda^{2\alpha}u \|^{2}_{L^{2}}+\f{1}{2}\| \Lambda^{3\alpha}u \|^{2}_{L^{2}} \leq C\|\Lambda^{ 2\alpha}u \|_{L^{2}}
^{\f{2( 10\alpha-5)}
{8\alpha-5}} +C
\|\Lambda^{\alpha}u \|^{\f{2( 10\alpha-5)}{ 6\alpha-5}}_{L^{2}}.
\ea
\ee
From \eqref{4.21v1}, we know that
$$\ba
\f{d}{dt}F_{0,1}+G_{0,1}\leq CF_{0,1}^{\f{ 8\alpha-5}{6\alpha-5}}.
\ea
$$
Consequently, a simple manipulation  leads to
\be\ba\label{4.24v1}
\f{6\alpha-5}{ 8\alpha-5}\f{d}{dt}F_{0,1}^{\f{ 8\alpha-5}{6\alpha-5}}+F^{\f{2\alpha}{6\alpha-5}}_{0,1}G_{0,1}\leq CF_{0,1}^{\f{ 10\alpha-5}{6\alpha-5}}.
\ea
\ee
Combining  \eqref{4.21} and \eqref{4.24}, we get
$$\ba
&\f{1}{2}\f{d}{dt}\B[\| \Lambda^{2\alpha}u \|^{2}_{L^{2}}+\f{6\alpha-5}{ 8\alpha-5}F_{0,1}^{\f{ 8\alpha-5}{6\alpha-5}}\B]+\f12\| \Lambda^{3\alpha}u \|^{2}_{L^{2}}+F^{\f{2\alpha}{6\alpha-5}}_{0,1}G_{0,1} \\
\leq& C\|\Lambda^{ 2\alpha}u \|_{L^{2}}
^{\f{2( 10\alpha-5)}
{8\alpha-5}} +C
\|\Lambda^{\alpha}u \|^{\f{2( 10\alpha-5)}{ 6\alpha-5}}_{L^{2}}
+CF_{0,1}^{\f{ 10\alpha-5}{6\alpha-5}}\\
\leq&C \|\Lambda^{ 2\alpha}u \|_{L^{2}}
^{\f{2( 10\alpha-5)}
{8\alpha-5}}  +CF_{0,1}^{\f{ 10\alpha-5}{6\alpha-5}},
\ea
$$
where the definition of $F_{0,1}$ in \eqref{4.16} was used.

The latter  inequality implies that
\be\ba\label{4.24}
\f{d}{dt}F_{0,2}+G_{0,2}\leq CF_{0,2}^{\f{ 10\alpha-5}{8\alpha-5}}.
\ea
\ee
By means of Lemma \ref{hofflemma} and  $F_{0,2} \in L^{\f{ 6 \alpha-5}{8 \alpha-5}}$, we conclude by the  inequality \eqref{4.24}  that
$$
G_{0,2}\in L^{\f{ 6\alpha-5}{10\alpha-5}},
$$
which means that
\be\| \Lambda^{3\alpha}u \|^{2}_{L^{2}}\in L^{\f{ 6\alpha-5}{10\alpha-5}}.
\ee
We have proved \eqref{4.13} with $p=0, q=2$.
Next, we focus on \eqref{4.13} with $p=q=1$.
Indeed, we drive from \eqref{4.21} that
 $\|\Lambda^{\alpha}u  _{t}\|_{L^{2}}^{2}\in L^{\f{6\alpha-5}{10\alpha-5}}$ and $F_{0,1} \in L^1$, which leads to
 \be\label{4.29}
 F_{1,1} \in L^{\f{6\alpha-5}{10\alpha-5}}.
 \ee
Along the line of \eqref{4.24}, we arrive at
\be\ba\label{4.26}
\f{6\alpha-5}{  10\alpha-5 }\f{d}{dt}F_{0,1}^{\f{ 10\alpha-5}{6\alpha-5}}+F^{\f{4\alpha}{6\alpha-5}}_{0,1}G_{0,1}\leq CF_{0,1}^{\f{ 12\alpha-5}{6\alpha-5}}.
\ea
\ee
A combination of \eqref{3.20} and \eqref{4.26} leads to
$$\ba
& \f{d}{dt}\B[\|\Lambda^{ \alpha} u _{t}\|^{2}_{L^{2}}+\f{6\alpha-5}{  10\alpha-5 }F_{0,1}^{\f{ 10\alpha-5}{6\alpha-5}}\B]+\f12\|\Lambda^{2 \alpha}u _{t}\|_{L^{2}}^{2}+F^{\f{4\alpha}{6\alpha-5}}_{0,1}G_{0,1}\\
\leq & C\|\Lambda^{\alpha}u_{t}\|_{L^{2}}^{\f{2(12\alpha-5)}{10\alpha-5}}
+C\|\Lambda^{\alpha}u \|_{L^{2}}^{\f{2(12\alpha-5)}{6\alpha-5}}
+\eta\|\Lambda^{ \alpha}u \|^{\f{8\alpha}{ 6\alpha-5}}_{L^{2}} \|\Lambda^{2 \alpha}u \|_{L^{2}}^{2}+CF_{0,1}^{\f{ 12\alpha-5}{6\alpha-5}}\\
\leq & C\|\Lambda^{\alpha}u_{t}\|_{L^{2}}^{\f{2(12\alpha-5)}{10\alpha-5}}
 +C\eta F^{\f{4\alpha}{6\alpha-5}}_{0,1}G_{0,1}+CF_{0,1}^{\f{ 12\alpha-5}{6\alpha-5}},
\ea
$$
where we have used
$$F_{0,1}^{\f{ 12\alpha-5}{6\alpha-5}}\leq CF_{1,1}^{\f{ 12\alpha-5}{10\alpha-5}}.$$
Thus,
$$\ba
&\f{1}{2}\f{d}{dt}\B[\|\Lambda^{ \alpha} u _{t}\|^{2}_{L^{2}}+F_{0,1}^{\f{ 10\alpha-5}{6\alpha-5}}\B]+\f12\|\Lambda^{2 \alpha}u _{t}\|_{L^{2}}^{2}+\f12F^{\f{4\alpha}{6\alpha-5}}_{0,1}G_{0,1} \\
\leq & C\|\Lambda^{\alpha}u_{t}\|_{L^{2}}^{\f{2(12\alpha-5)}{10\alpha-5}}
 +CF_{1,1}^{\f{ 12\alpha-5}{10\alpha-5}},
\ea
$$
which means
\be\label{4.33v1}
\f{d}{dt}F_{1,1}+G_{1,1}\leq
CF_{1,1}^{\f{ 12\alpha-5}{10\alpha-5}}.
\ee
This together with Lemma \ref{hofflemma} implies that
\be
G_{1,1}\in L^{\f{6\alpha-5}{12\alpha-5}}, \|\Lambda^{2 \alpha}u _{t}\|_{L^{2}}^{2}\in L^{\f{6\alpha-5}{12\alpha-5}}.
\ee
It remains to show \eqref{4.13} with $p=2, q=0$. To this end, we deduce from
\eqref{3.31v1} with $r=2$ that
\be\ba\label{4.30}
&\f{1}{2}\f{d}{dt}\| u _{t}^{(2)}\|^{2}_{L^{2}}+\|\Lambda^{ \alpha}u _{t}^{(2)}\|_{L^{2}}^{2}  \\
\leq&C\|u_{t} \|
^{\f{2(14\alpha-5)}{8\alpha-5}}
_{L^{2}}+\|u_{t}^{(2)} \|
^{\f{2(14\alpha-5)}{12\alpha-5}}
_{L^{2}}+C\|\Lambda^{\alpha}u_{t}  \|_{L^{2}}^{\f{2(14\alpha-5)}{ 10\alpha-5}}+C\|\Lambda^{\alpha}u  \|_{L^{2}}^{\f{2(14\alpha-5)}{ 6\alpha-5}}\\&+2\eta\|\Lambda^{\alpha}u_{t} \|_{L^{2}}^{\f{4\alpha }{ 10\alpha-5}}
\|\Lambda^{2\alpha}u_{t} \|^{2}_{L^{2}}+2\eta\|\Lambda^{\alpha}u  \|_{L^{2}}^{\f{12\alpha }{  6\alpha-5}}
\|\Lambda^{2\alpha}u   \|^{2}_{L^{2}}.
\ea
\ee
It follows from \eqref{4.21v1} and
\eqref{4.33v1} that
$$\ba \f{d}{dt}F_{0,1}+G_{0,1}\leq CF_{0,1}^{\f{ 8\alpha-5}{6\alpha-5}},\\
\f{d}{dt}F_{1,1}+G_{1,1}\leq
CF_{1,1}^{\f{ 12\alpha-5}{10\alpha-5}},\\
\f{d}{dt}F_{ 1,0}+G_{1,0}\leq CF_{ 1,0}^{\f{ 10\alpha-5}{8\alpha-5}},\ea$$
which turns out that
\be\ba\label{4.31}
\f{6 \alpha-5}{ 12 \alpha-5}\f{d}{dt}F_{0,1}^{\f{ 12 \alpha-5}{6 \alpha-5}}+F_{0,1}^{\f{ 6 \alpha }{6 \alpha-5}}G_{0,1}\leq CF_{0,1}^{\f{ 14\alpha-5}{6\alpha-5}}\\
\f{10 \alpha-5}{ 12 \alpha-5}\f{d}{dt}F_{1,1}^{\f{ 12 \alpha-5}{10 \alpha-5}}+F_{1,1}^{\f{2 \alpha }{10 \alpha-5}}G_{1,1}\leq
CF_{1,1}^{\f{ 14\alpha-5}{10\alpha-5}}\\
\f{8 \alpha-5}{ 12 \alpha-5}\f{d}{dt}F_{1,0}^{\f{ 12 \alpha-5}{8 \alpha-5}}+F_{ 1,0}^{\f{ 4\alpha }{8\alpha-5}}G_{1,0}\leq CF_{ 1,0}^{\f{ 14\alpha-5}{8\alpha-5}}.\ea\ee
Summarizing \eqref{4.30} and \eqref{4.31}, we get
\be\ba
&\f{d}{dt}\B[\| u _{t}^{(2)}\|^{2}_{L^{2}}+\f{8 \alpha-5}{ 12 \alpha-5}F_{1,0}^{\f{ 12 \alpha-5}{8 \alpha-5}}+\f{6 \alpha-5}{ 12 \alpha-5}
F_{0,1}^{\f{ 12 \alpha-5}{6 \alpha-5}}+\f{10 \alpha-5}{ 12 \alpha-5}F_{1,1}^{\f{ 12 \alpha-5}{10 \alpha-5}}\B]\\&+\f{1}{2}\|\Lambda^{ \alpha}u _{t}^{(2)}\|_{L^{2}}^{2} +F_{0,1}^{\f{ 6 \alpha }{6 \alpha-5}}G_{0,1}+F_{1,1}^{\f{2 \alpha }{10 \alpha-5}}G_{1,1}\\
 \leq&C\|u_{t} \|
^{\f{2(14\alpha-5)}{8\alpha-5}}
_{L^{2}}+C\|u_{t}^{(2)} \|
^{\f{2(14\alpha-5)}{12\alpha-5}}
_{L^{2}}+C\|\Lambda^{\alpha}u_{t}  \|_{L^{2}}^{\f{2(14\alpha-5)}{ 10\alpha-5}}\\&+C\|\Lambda^{\alpha}u  \|_{L^{2}}^{\f{2(14\alpha-5)}{ 6\alpha-5}}+2\eta\|\Lambda^{\alpha}u_{t} \|_{L^{2}}^{\f{4\alpha }{ 10\alpha-5}}
\|\Lambda^{2\alpha}u_{t} \|^{2}_{L^{2}}+2\eta\|\Lambda^{\alpha}u  \|_{L^{2}}^{\f{12\alpha }{  6\alpha-5}}
\|\Lambda^{2\alpha}u   \|^{2}_{L^{2}}\\
&+CF_{0,1}^{\f{ 14\alpha-5}{6\alpha-5}}+CF_{1,1}^{\f{ 14\alpha-5}{10\alpha-5}}+CF_{ 1,0}^{\f{ 14\alpha-5}{8\alpha-5}}\\
 \leq& C\|u_{t}^{(2)} \|
^{\f{2(14\alpha-5)}{12\alpha-5}}
_{L^{2}} +CF_{0,1}^{\f{ 14\alpha-5}{6\alpha-5}}+CF_{1,1}^{\f{ 14\alpha-5}{10\alpha-5}}+CF_{ 1,0}^{\f{ 14\alpha-5}{8\alpha-5}}\\& +2\eta\|\Lambda^{\alpha}u_{t} \|_{L^{2}}^{\f{4\alpha }{ 10\alpha-5}}
\|\Lambda^{2\alpha}u_{t} \|^{2}_{L^{2}}+2\eta\|\Lambda^{\alpha}u  \|_{L^{2}}^{\f{12\alpha }{  6\alpha-5}}
\|\Lambda^{2\alpha}u   \|^{2}_{L^{2}},
\ea
\ee
where the definition in \eqref{4.16} was used.\\
Taking $\eta>0$ sufficient small, we get
\be\label{4.33}
\f{d}{dt}F_{2,0}+G_{2,0}\leq CF_{2,0}^{\f{ 14\alpha-5}{12\alpha-5}}.
\ee
From last section, we know that $\| u _{t}^{(2)}\|^{2}_{L^{2}}\in  L^{\f{6 \alpha-5}{ 12 \alpha-5}}.$ Moreover, we derive from
\eqref{4.21} and \eqref{4.29} that
$$
F_{0,1} \in L^1,F_{ 1,0} \in   L^{\f{ 6\alpha-5 }{ 8\alpha-5}}, F_{1,1} \in L^{\f{6\alpha-5}{10\alpha-5}},
$$
which turns out that
\be
F_{2,0}\in L^{\f{6 \alpha-5}{ 12 \alpha-5}}.
\ee
Then we apply Lemma \ref{hofflemma} to estimate \eqref{4.33} to obtain
$$
\|\Lambda^{ \alpha}u _{t}^{(2)}\|_{L^{2}}^{2}\in L^{\f{6 \alpha-5}{ 14 \alpha-5}}.
$$
Hence, we have verified \eqref{4.13} for $p+q=1$ and $p+q=2$.\\
Now, we suppose that \eqref{4.13} is valid for
 \be\label{4.41v1}
 p+q\leq  k-2= m+n -2.
 \ee
It is enough to show \eqref{4.13} for $p+q=k-1= m+n -1$. We will  divide the deduction of $F_{i,m+n-i-1}$  into four cases:
\begin{enumerate}[(1)]
      \item
Case 1: $i=0$;
 \item Case 2: $1\leq i\leq m+n -3$;
 \item Case 3: $i= m+n -2$;
\item Case 4: $i= m+n -1$.\\
\end{enumerate}
It worth pointing out that the derivation of four cases is in turn. The last case rests on the result of Case 3.

\underline{Case 1}. We notice that
\be\ba\label{4.36v1}
&F_{0, m+n -1}=\|\Lambda^{(m+n -1)\alpha}u \|_{L^{2}}^{2}+\f{6 \alpha-5}{ 2(m+n -1)\alpha+4 \alpha-5}
F_{0,1}^{\f{ 2(m+n -1)\alpha+4 \alpha-5}{6 \alpha-5}},\\
&G_{0,m+n -1}=\f12 \|\Lambda^{(m+n  )\alpha}u \|_{L^{2}}^{2}+\f12
F_{0,1}^{\f{2(m+n -1)\alpha-2\alpha }{6 \alpha-5}}G_{0,1}.
\ea\ee
By means of inductive hypothesis \eqref{4.41v1} and $\eqref{4.13}_{3}$, we know that $\|\Lambda^{(m+n -1)\alpha}u \|_{L^{2}}^{2}\in  L^{\f{6 \alpha-5}{ 2(m+n -1)\alpha+4 \alpha-5}}.$ Hence,  owing to $F_{0,1} \in L^1 $, we get $F_{0, m+n -1}\in  L^{\f{6 \alpha-5}{ 2(m+n -1)\alpha+4 \alpha-5}}.$\\
By virtue of \eqref{4.21v1}, we find
$$\ba \f{d}{dt}F_{0,1}+G_{0,1}\leq CF_{0,1}^{\f{ 8\alpha-5}{6\alpha-5}}.
\ea$$
Consequently, we readily check
\be\ba\label{4.38} \f{d}{dt}F_{0,1}^{\f{ 2(m+n -1)\alpha+4 \alpha-5}{6 \alpha-5}}+F_{0,1}^{\f{2(m+n -1)\alpha-2\alpha }{6 \alpha-5}}G_{0,1}\leq CF_{0,1}^{\f{   2(m+n -1)\alpha+6\alpha-5  }{6\alpha-5}}.
\ea\ee
In addition, thanks to \eqref{4.15v1}, we write
$$\ba
& \f{d}{dt} \| \Lambda^{(m+n -1)\alpha}u \|^{2}_{L^{2}}+\f12\| \Lambda^{ (m+n ) \alpha}u \|^{2}_{L^{2}}\\
\leq& C\|\Lambda^{ (m+n -1) \alpha}u \|_{L^{2}}
^{\f{2[ 2(m+n)\alpha+4\alpha-5]}
{ 2(m+n -1)\alpha+4\alpha-5}} +
C\|\Lambda^{\alpha}u \|^{\f{2[ 2(m+n)\alpha+4\alpha-5]}{  2 \alpha+4\alpha-5}}_{L^{2}}\\
\leq& C\|\Lambda^{ (m+n -1) \alpha}u \|_{L^{2}}
^{\f{2[ 2(m+n)\alpha+4\alpha-5]}
{ 2(m+n -1)\alpha+4\alpha-5}} +
CF_{0,1}^{\f{2[ 2(m+n)\alpha+4\alpha-5]}{  2 \alpha+4\alpha-5}}.
\ea
$$
Combining the latter relationship and \eqref{4.38}, we see that
\be\ba
& \f{d}{dt} \B[\| \Lambda^{(m+n -1)\alpha}u \|^{2}_{L^{2}}+\f{6 \alpha-5}{ 2(m+n -1)\alpha+4 \alpha-5}F_{0,1}^{\f{ 2(m+n -1)\alpha+4 \alpha-5}{6 \alpha-5}}\B]\\&+\f12\| \Lambda^{ (m+n ) \alpha}u \|^{2}_{L^{2}}+F_{0,1}^{\f{2(m+n -1)\alpha-2\alpha }{6 \alpha-5}}G_{0,1} \\
\leq &C\|\Lambda^{ (m+n -1) \alpha}u \|_{L^{2}}
^{\f{2[ 2(m+n)\alpha+4\alpha-5]}
{ 2(m+n -1)\alpha+4\alpha-5}} +C
F_{0,1}^{\f{  2(m+n)\alpha+4\alpha-5 }{  2 \alpha+4\alpha-5}}.
\ea
\ee
Thanks to the definition of $F_{0, m+n -1}$ in \eqref{4.36v1}, we find
$$
F_{0,1}^{\f{  2(m+n)\alpha+4\alpha-5 }{  2 \alpha+4\alpha-5}}\leq
CF_{0,m+n-1}^{\f{  2(m+n)\alpha+4\alpha-5 }{  2(m+n -1)\alpha+4 \alpha-5}}.
$$
Hence, we remark that
$$\ba \f{d}{dt}F_{0,m+n-1}+G_{0,m+n-1}\leq F_{0,m+n-1}^{\f{  2(m+n)\alpha+4\alpha-5 }
{ 2(m+n -1)\alpha+4\alpha-5}}.
\ea$$
Lemma \ref{hofflemma} and $F_{0, m+n -1}\in  L^{\f{6 \alpha-5}{ 2(m+n -1)\alpha+4 \alpha-5}} $  ensure that
$$G_{0,m+n-1}\in L^{\f{6\alpha-5}{ 2(m+n)\alpha+4\alpha-5}}, \| \Lambda^{ (m+n ) \alpha}u \|^{2}_{L^{2}}\in L^{\f{6\alpha-5}{ 2(m+n)\alpha+4\alpha-5}}.$$
The above analysis verifies  \eqref{4.13} with $p+q=k-1= m+n -1$ for Case 1.

\underline{Case 2}.  $1\leq i<m+n-2$.\\
Before going further, we write
\be\ba\label{4.41}
F_{i,m+n-i-1}= & \|\Lambda^{(m+n-i-1)\alpha}u_{t}^{(i)}\|_{L^{2}}^{2}\\&+ \sum_{\ell=0}^{i-1}\f{4\ell\alpha+2(m+n-i-1)\alpha+4 \alpha-5}{4i\alpha+2(m+n-i-1)\alpha+4 \alpha-5}
F_{\ell,m+n-i-1}^{\f{4i\alpha+2(m+n-i-1)\alpha+4 \alpha-5}{4\ell\alpha+2(m+n-i-1)\alpha+4 \alpha-5}}\\&+\sum_{\ell=0}^{i }\f{4\ell\alpha+6 \alpha-5}{4i\alpha+2(m+n-i-1)\alpha+4 \alpha-5}
F_{\ell,1}^{\f{4i\alpha+2(m+n-i-1)\alpha+4 \alpha-5}{4\ell\alpha+6 \alpha-5}},\\
G_{i,m+n-i-1}= &\f12\|\Lambda^{(m+n-i )\alpha}u_{t}^{(i)}\|_{L^{2}}^{2}+\f12\sum_{\ell=0}^{i-1}
F_{\ell,m+n-i-1}^{\f{4(i-\ell)\alpha }{4\ell\alpha+2(m+n-i-1)\alpha+4 \alpha-5}}G_{\ell,m+n-i-1}.
\ea\ee
According to inductive hypothesis \eqref{4.41v1} and $\eqref{4.13}_{3}$,  we assert that $\|\Lambda^{(m+n-i-1)\alpha}u_{t}^{(i)}\|_{L^{2}}^{2}\in L^{\f{6 \alpha-5}{4i\alpha+2(m+n-i-1)\alpha+4 \alpha-5}} $. Moreover, using the fact that $\ell+m+n-i-1\leq m+n -2$ and $\ell+1\leq i+1\leq m+n-2$, by   the inductive hypotheses  \eqref{4.41v1},  we arrive at $F_{\ell,m+n-i-1}\in L^{\f{6\alpha-5}{4\ell\alpha+2(m+n-i-1)\alpha+4\alpha-5}}$     and   $F_{\ell,1}\in L^{\f{6\alpha-5}{4\ell\alpha +6\alpha-5}}$. Hence, we have
\be\label{4.48}
F_{i,m+n-i-1}\in L^{\f{6 \alpha-5}{4i\alpha+2(m+n-i-1)\alpha+4 \alpha-5}}.
\ee
Making use of \eqref{4.41v1}, we get
\be\ba
&\f{d}{dt}F_{\ell,m+n-i-1}+G_{\ell,m+n-i-1}\leq CF_{\ell,m+n-i-1}^{\f{4\ell\alpha+2(m+n-i-1)\alpha+6 \alpha-5}{4\ell\alpha+2(m+n-i-1)\alpha+4 \alpha-5}},\ell \leq i-1,\\
&\f{d}{dt}F_{\ell,1}+G_{\ell,1}\leq CF_{\ell,1}^{\f{4\ell\alpha+ 8 \alpha-5}{4\ell\alpha+6 \alpha-5}},\ell \leq i.
\ea\ee
Some tedious manipulation yields
\be\ba\label{4.44}
 &c_{\ell,m+n-i-1}\f{d}{dt}F_{\ell,m+n-i-1}^{\f{4i\alpha+2(m+n-i-1)\alpha+4 \alpha-5}{4\ell\alpha+2(m+n-i-1)\alpha+4 \alpha-5}}+
F_{\ell,m+n-i-1}^{\f{4(i-\ell)\alpha }{4\ell\alpha+2(m+n-i-1)\alpha+4 \alpha-5}}G_{\ell,m+n-i-1}\leq C F_{\ell,m+n-i-1}^{\f{4i\alpha+2(m+n-i-1)\alpha+6 \alpha-5}{4\ell\alpha+2(m+n-i-1)\alpha+4 \alpha-5}},\\
& c_{\ell,1}\f{d}{dt}
F_{\ell,1}^{\f{4i\alpha+2(m+n-i-1)\alpha+4 \alpha-5}{4\ell\alpha+6 \alpha-5}}
+
F_{\ell,1}^{\f{4(i-\ell)\alpha+2(m+n-i-1)\alpha-2 \alpha}{4\ell\alpha+6 \alpha-5}}G_{\ell,1}\leq C F_{\ell,1}^{\f{4i\alpha+2(m+n-i-1)\alpha+6 \alpha-5}{4\ell\alpha+6 \alpha-5}},
\ea\ee
where
$$\ba
c_{\ell,m+n-i-1}=\f{4\ell\alpha+2(m+n-i-1)\alpha+4 \alpha-5}{4i\alpha+2(m+n-i-1)\alpha+4 \alpha-5},\\
c_{\ell,1}=\f{4\ell\alpha+6 \alpha-5}{4i\alpha+2(m+n-i-1)\alpha+4 \alpha-5}.
\ea$$
We deduce from
\eqref{4.13v2} that
\be\ba\label{4.4333}
&\f{d}{dt} \| \Lambda^{(m+n-i-1)\alpha}u _{t}^{(i)}\|^{2}_{L^{2}}+\f{1}{2}\| \Lambda^{(m+n-i )\alpha}u _{t}^{(r)}\|^{2}_{L^{2}}\\\leq&C\sum_{j }^{i-1} \|\Lambda^{ (m+n-i-1) \alpha}u_{t}^{(j)}\|_{L^{2}}
^{\f{2(4i\alpha+2(m+n-i)\alpha+4\alpha-5)}
{4j\alpha+2(m+n-i-1)\alpha+4\alpha-5}}+C\|\Lambda^{ (m+n-i-1) \alpha}u_{t}^{(i)}\|_{L^{2}}
^{\f{2(4i\alpha+2(m+n-i)\alpha+4\alpha-5)}
{4i\alpha+2(m+n-i-1)\alpha+4\alpha-5}}\\&+C\sum_{j }^{i}
\|\Lambda^{\alpha}u^{(i-j)}_{t} \|^{\f{2(4i\alpha+2(m+n-i)\alpha+4\alpha-5)}{4(i-j)\alpha+2 \alpha+4\alpha-5}}_{L^{2}}\\&+\eta\sum_{j>0 }^{i}\|\Lambda^{(m+n-i-1)\alpha}u^{(i-j )}_{t}\|^{\f{8\alpha j}{ 4(r-j)\alpha+2 (m+n-i-1)\alpha+4\alpha-5}}_{L^{2}}\|\Lambda^{(m+n-i)\alpha} u^{(i-j )}_{t}\|_{L^{2}}^{2}\\\leq&C\sum_{j }^{i-1} F_{j,m+n-i-1}
^{\f{ 4i\alpha+2(m+n-i)\alpha+4\alpha-5}
{4j\alpha+2(m+n-i-1)\alpha+4\alpha-5}}+C\|\Lambda^{ (m+n-i-1) \alpha}u_{t}^{(i)}\|_{L^{2}}
^{\f{2(4i\alpha+2(m+n-i)\alpha+4\alpha-5)}
{4i\alpha+2(m+n-i-1)\alpha+4\alpha-5}}\\&+C\sum_{j=0 }^{i}
F_{j,1}^{\f{ 4i\alpha+2(m+n-i)\alpha+4\alpha-5}{4j\alpha+2 \alpha+4\alpha-5}} +\eta\sum_{j=0 }^{i-1}F_{j,m+n-i-1}^{\f{4\alpha (i-j)}{ 4j\alpha+2 (m+n-i-1)\alpha+4\alpha-5}}G_{j,m+n-i-1},
\ea\ee
where the definitions of  $F_{j,1}, F_{j,m+n-i-1}, G_{j,m+n-i-1}$ in \eqref{4.41} were used.\\
In view of \eqref{4.44} and    \eqref{4.4333}, we infer that
\be\ba\label{4.46}
&\f{d}{dt} \B[\| \Lambda^{(m+n-i-1)\alpha}u _{t}^{(i)}\|^{2}_{L^{2}}+\sum_{\ell=0}^{i-1 }c_{\ell,m+n-i-1}F_{\ell,m+n-i-1}^{\f{4i\alpha+2(m+n-i-1)\alpha+4 \alpha-5}{4\ell\alpha+2(m+n-i-1)\alpha+4 \alpha-5}}\\&+\sum_{\ell=0}^{i } c_{\ell,1}F_{\ell,1}^{\f{4i\alpha+2(m+n-i-1)\alpha+4 \alpha-5}{4\ell\alpha+6 \alpha-5}}
\B]\\&+\f{1}{2}\| \Lambda^{(m+n-i )\alpha}u _{t}^{(r)}\|^{2}_{L^{2}}+F_{\ell,m+n-i-1}^{\f{4(i-\ell)\alpha }{4\ell\alpha+2(m+n-i-1)\alpha+4 \alpha-5}}G_{\ell,m+n-i-1}  \\\leq&C\sum_{j }^{i-1} F_{j,m+n-i-1}
^{\f{ 4i\alpha+2(m+n-i)\alpha+4\alpha-5}
{4j\alpha+2(m+n-i-1)\alpha+4\alpha-5}}+C\|\Lambda^{ (m+n-i-1) \alpha}u_{t}^{(i)}\|_{L^{2}}
^{\f{2(4i\alpha+2(m+n-i)\alpha+4\alpha-5)}
{4i\alpha+2(m+n-i-1)\alpha+4\alpha-5}}\\&+C\sum_{j=0 }^{i}
F_{j,1}^{\f{ 4i\alpha+2(m+n-i)\alpha+4\alpha-5}{4j\alpha+2 \alpha+4\alpha-5}} +\eta\sum_{j=0 }^{i-1}F_{j,m+n-i-1}^{\f{4\alpha (i-j)}{ 4j\alpha+2 (m+n-i-1)\alpha+4\alpha-5}}G_{j,m+n-i-1}.
\ea\ee
On account of \eqref{4.41}, we discover that
 $$\ba&\sum_{j }^{i-1} F_{j,m+n-i-1}
^{\f{ 4i\alpha+2(m+n-i)\alpha+4\alpha-5}
{4j\alpha+2(m+n-i-1)\alpha+4\alpha-5}}\leq C\sum_{j }^{i-1} F_{i,m+n-i-1}^{\f{4i\alpha+2(m+n-i)\alpha+4\alpha-5}{4i\alpha+2(m+n-i-1)\alpha+4 \alpha-5}}\leq C F_{i,m+n-i-1}^{\f{4i\alpha+2(m+n-i)\alpha+4\alpha-5}{4i\alpha+2(m+n-i-1)\alpha+4 \alpha-5}},\\&\sum_{j=0 }^{i}
F_{j,1}^{\f{  4i\alpha+2(m+n-i)\alpha+4\alpha-5 }{4j\alpha+2 \alpha+4\alpha-5}}
\leq C\sum_{j=0 }^{i}
F_{i,m+n-i-1}^{\f{4i\alpha+2(m+n-i)\alpha+4\alpha-5}{4i\alpha+2(m+n-i-1)\alpha+4 \alpha-5}}\leq C
F_{i,m+n-i-1}^{\f{4i\alpha+2(m+n-i)\alpha+4\alpha-5}{4i\alpha+2(m+n-i-1)\alpha+4 \alpha-5}}.\ea$$
Inserting this into \eqref{4.46}, we end up with
\be\label{4.45}
\f{d}{dt}F_{i,m+n-i-1}+G_{i,m+n-i-1}\leq
CF_{i,m+n-i-1}^{\f{4i\alpha+2(m+n-i)\alpha+4\alpha-5}{4i\alpha+2(m+n-i-1)\alpha+4 \alpha-5}}.
\ee
Using
\eqref{4.48} and Lemma \ref{hofflemma}, we conclude by \eqref{4.45} that
\be
G_{i,m+n-i-1}\in L^{\f{6\alpha-5}{4i\alpha+2(m+n-i)\alpha+4\alpha-5}}, \| \Lambda^{(m+n-i )\alpha}u _{t}^{(i)}\|^{2}_{L^{2}}\in L^{\f{6\alpha-5}{4i\alpha+2(m+n-i)\alpha+4\alpha-5}}.
\ee
The validity of  \eqref{4.13} with $p+q=k-1= m+n -1$ for Case 2 is thus confirmed.

\underline{Case 3}. $i=m+n-2$.\\
In this case, we set
\be\ba\label{4.49}
&F_{ m+n-2,1}= \|\Lambda^{ \alpha}u_{t}^{(m+n-2)}\|_{L^{2}}^{2}+
\sum_{\ell=0}^{m+n-3}\f{4\ell\alpha+6\alpha-5}{4(m+n-2)\alpha+6 \alpha-5}
F_{\ell,1}^{\f{4(m+n-2)\alpha+6 \alpha-5}{4\ell\alpha+6\alpha-5}}, \\
&G_{m+n-2,  1}= \f12 \|\Lambda^{2\alpha}u_{t}^{(m+n-2)}\|_{L^{2}}^{2}+\f12\sum_{\ell=0}^{m+n-3}
F_{\ell,1}^{\f{4(i-\ell)\alpha }{4\ell\alpha+2 \alpha+4 \alpha-5}}G_{\ell,1}.
\ea\ee
It is apparent from induction hypotheses  \eqref{4.41v1} and $\eqref{4.13}_{3}$  that $\|\Lambda^{ \alpha}u_{t}^{(m+n-2)}\|_{L^{2}}^{2}\in L^{\f{6\alpha-5}{4(m+n-2)\alpha +6 \alpha-5}}$.
In addition, since $\ell+1\leq m+n-2$, we conclude by \eqref{4.41v1} that   $F_{\ell,1}\in L^{\f{6\alpha-5}{4\ell\alpha +6\alpha-5}}$.
Hence, we find
\be\label{4.56}
F_{ m+n-2,1}\in   L^{\f{6\alpha-5}{4(m+n-2)\alpha +6 \alpha-5}}.
\ee
Then we deduce from
\eqref{4.13v2} that
\be\ba
&\f{d}{dt} \| \Lambda^{ \alpha}u _{t}^{(m+n-2)}\|^{2}_{L^{2}}+\f{1}{2}\| \Lambda^{2\alpha}u _{t}^{(m+n-2)}\|^{2}_{L^{2}}\\\leq&C\sum_{j }^{m+n-3} \|\Lambda^{ \alpha}u_{t}^{(j)}\|_{L^{2}}
^{\f{2(4(m+n-2)\alpha+8\alpha-5)}
{4j\alpha+6\alpha-5}}+C\|\Lambda^{   \alpha}u_{t}^{(m+n-2)}\|_{L^{2}}
^{\f{2(4(m+n-2)\alpha+8\alpha-5)}
{4(m+n-2)\alpha+6\alpha-5}}\\&+C\sum_{j }^{m+n-2}
\|\Lambda^{\alpha}u^{(m+n-2-j)}_{t} \|^{\f{2(4(m+n-2)\alpha+8\alpha-5)}{4(m+n-2-j)\alpha
+6\alpha-5}}_{L^{2}}\\&
+\eta\sum_{j>0 }^{m+n-2}
\|\Lambda^{ \alpha}u^{(m+n-2-j )}_{t}\|^{\f{8\alpha j}{ 4(m+n-2-j)\alpha+6\alpha-5}}_{L^{2}}\|\Lambda^{2\alpha} u^{(m+n-2-j )}_{t}\|_{L^{2}}^{2}\\
\leq&C\sum_{j }^{m+n-3} \|\Lambda^{ \alpha}u_{t}^{(j)}\|_{L^{2}}
^{\f{2(4(m+n-2)\alpha+8\alpha-5)}
{4j\alpha+6\alpha-5}}+C\|\Lambda^{   \alpha}u_{t}^{(m+n-2)}\|_{L^{2}}
^{\f{2(4(m+n-2)\alpha+8\alpha-5)}
{4(m+n-2)\alpha+6\alpha-5}} \\&+\eta\sum_{j>0 }^{m+n-2}
\|\Lambda^{ \alpha}u^{(m+n-2-j )}_{t}\|^{\f{8\alpha j}{ 4(m+n-2-j)\alpha+6\alpha-5}}_{L^{2}}\|\Lambda^{2\alpha} u^{(m+n-2-j )}_{t}\|_{L^{2}}^{2}\\\leq&
C\sum_{j }^{m+n-3} F_{j,1}
^{\f{ 4(m+n-2)\alpha+8\alpha-5 }
{4j\alpha+6\alpha-5}}+C\|\Lambda^{   \alpha}u_{t}^{(m+n-2)}\|_{L^{2}}
^{\f{2(4(m+n-2)\alpha+8\alpha-5)}
{4(m+n-2)\alpha+6\alpha-5}} \\&+\eta\sum_{j=0 }^{m+n-3} F_{j,1}^{\f{4\alpha (m+n-2-j ) }{ 4j\alpha+6\alpha-5}}G_{j,1}.
\ea\ee
Owing to \eqref{4.49}, we get
$$\sum_{j }^{m+n-3}F_{j,1}
^{\f{ 4(m+n-2)\alpha+8\alpha-5 }
{4j\alpha+6\alpha-5}}\leq CF_{ m+n-2,1}^{\f{ 4(m+n-2)\alpha+8\alpha-5 }
{4(m+n-2)\alpha+6 \alpha-5}},$$
which turns out that
\be\ba\label{4.52}
&\f{d}{dt} \| \Lambda^{ \alpha}u _{t}^{(m+n-2)}\|^{2}_{L^{2}}+\f{1}{2}\| \Lambda^{2\alpha}u _{t}^{(m+n-2)}\|^{2}_{L^{2}}
\\\leq&
CF_{ m+n-2,1}^{\f{ 4(m+n-2)\alpha+8\alpha-5 }
{4(m+n-2)\alpha+6 \alpha-5}}+C\|\Lambda^{   \alpha}u_{t}^{(m+n-2)}\|_{L^{2}}
^{\f{2(4(m+n-2)\alpha+8\alpha-5)}
{4(m+n-2)\alpha+6\alpha-5}} +\eta\sum_{j=0 }^{m+n-3} F_{j,1}^{\f{4\alpha (m+n-2-j ) }{ 4j\alpha+6\alpha-5}}G_{j,1}.
\ea\ee
Since $j+1\leq m+n-2$,   by     the induction hypotheses \eqref{4.41v1}, we observe that
$$
\f{d}{dt}F_{j,1}+G_{j,1}\leq CF_{j,1}^{\f{4j\alpha+ 8 \alpha-5}{4j\alpha+6 \alpha-5}},
$$
from which it follows that
\be\label{4.53}
\f{d}{dt}
F_{j,1}^{\f{4(m+n-2)\alpha+6 \alpha-5}{4j\alpha+6\alpha-5}}+ F_{j,1}^{\f{4\alpha (m+n-2-j ) }{ 4j\alpha+6\alpha-5}}G_{j,1}\leq CF_{j,1}^{\f{4\alpha (m+n-2)+ 8 \alpha-5}{4j\alpha+6 \alpha-5}}.
\ee
It follows from \eqref{4.52} and \eqref{4.53} that
\be\ba
& \f{d}{dt}\B[ \| \Lambda^{ \alpha}u _{t}^{(m+n-2)}\|^{2}_{L^{2}}+ \sum_{j=0 }^{m+n-3}\f{4j\alpha+6\alpha-5}{4(m+n-2)\alpha+6 \alpha-5}F_{j,1}^{\f{4(m+n-2)\alpha+6 \alpha-5}{4j\alpha+6\alpha-5}}\B]
\\&+\f12\| \Lambda^{2\alpha}u _{t}^{(m+n-2)}\|^{2}_{L^{2}}+ \sum_{j=0 }^{m+n-3}F_{j,1}^{\f{4\alpha (m+n-2-j ) }{ 4j\alpha+6\alpha-5}}G_{j,1}
\\\leq&
CF_{ m+n-2,1}^{\f{ 4(m+n-2)\alpha+8\alpha-5 }
{4(m+n-2)\alpha+6 \alpha-5}}+C\|\Lambda^{   \alpha}u_{t}^{(m+n-2)}\|_{L^{2}}
^{\f{2(4(m+n-2)\alpha+8\alpha-5)}
{4(m+n-2)\alpha+6\alpha-5}}  +\eta\sum_{j=0 }^{m+n-3} F_{j,1}^{\f{4\alpha (m+n-2-j ) }{ 4j\alpha+6\alpha-5}}G_{j,1},
\ea\ee
which means that
\be\ba\label{4.60}
&\f{d}{dt}F_{ m+n-2,1}+G_{ m+n-2,1}\leq CF_{ m+n-2,1}^{\f{ 4(m+n-2)\alpha+8\alpha-5 }
{4(m+n-2)\alpha+6 \alpha-5}}.
\ea\ee
By virtue of
Lemma \ref{hofflemma} and  \eqref{4.56}, we conclude by \eqref{4.60} that
$$G_{ m+n-2,1}\in L^{\f{6 \alpha-5}{4(m+n-2)\alpha+8\alpha-5 }}, \|\Lambda^{ 2\alpha}u _{t}^{(m+n-2)}\|_{L^{2}}^{2}\in L^{\f{6 \alpha-5}{4(m+n-2)\alpha+8\alpha-5 }}.$$
We show \eqref{4.13} with $p+q=k-1= m+n -1$ for Case 3.

\underline{Case 4}. $i=m+n-1$.\\
We denote
\be\ba\label{4.61v1}
&F_{ m+n-1,0}= \| u_{t}^{(m+n-1)}\|_{L^{2}}^{2}+\sum_{\ell=0}^{m+n-2}
c_{\ell,0}F_{\ell,0}^{\f{4(m+n-1)\alpha+4 \alpha-5}{4\ell\alpha+4\alpha-5}}+\sum_{\ell=0}^{m+n-2}c_{\ell,1}
F_{\ell,1}^{\f{4(m+n-1)\alpha+4 \alpha-5}{4\ell\alpha+6\alpha-5}}, \\
&G_{m+n-1,  0}= \f12\|\Lambda^{ \alpha}u_{t}^{(m+n-1)}\|_{L^{2}}^{2}+\f12\sum_{\ell=0}^{m+n-2}
c_{\ell,1}F_{\ell,1}^{\f{4(m+n-1-\ell)\alpha-2 \alpha }{4\ell\alpha+2 \alpha+4 \alpha-5}}G_{\ell,1},
\ea\ee
where
$$\ba
c_{\ell,0}=\f{4\ell\alpha+4\alpha-5}{4(m+n-1)\alpha+4 \alpha-5},\quad
c_{\ell,1}=\f{4\ell\alpha+6\alpha-5}{4(m+n-1)\alpha+4 \alpha-5}.
\ea$$
Last section guarantees that
$$\| u_{t}^{(m+n-1)}\|_{L^{2}}^{2}\in L^{\f{6 \alpha-5}{4(m+n-1)\alpha+4 \alpha-5}}.$$
As
$\ell+0\leq m+n-2$, we conclude by the induction hypotheses \eqref{4.41v1} that $F_{\ell,0}\in L^{\f{6\alpha-5}{4\ell\alpha +4\alpha-5}}$ and $F_{\ell,1}\in L^{\f{6\alpha-5}{4\ell\alpha +6\alpha-5}}$.
It is clear that
$$
\sum_{\ell=0}^{m+n-2}
F_{\ell,1}^{\f{4(m+n-1)\alpha+4 \alpha-5}{4\ell\alpha+6\alpha-5}}
=\sum_{\ell=0}^{m+n-3}
F_{\ell,1}^{\f{4(m+n-1)\alpha+4 \alpha-5}{4\ell\alpha+6\alpha-5}}+F_{m+n-2,1}^{\f{4(m+n-1)\alpha+4 \alpha-5}{4(m+n-2)\alpha+6\alpha-5}}.
$$
With the help of \eqref{4.56} in Case 3, we see that $F_{ m+n-2,1}\in   L^{\f{6\alpha-5}{4(m+n-2)\alpha +6 \alpha-5}}$.
Hence, we have
\be\label{4.61}
F_{ m+n-1,0}\in L^{\f{6 \alpha-5}{4(m+n-1)\alpha+4 \alpha-5}}.\ee
Combining
\eqref{3.31}  and \eqref{4.61}, we find
 \be\ba\label{4.62}
&\f{d}{dt}\| u _{t}^{(m+n-1)}\|^{2}_{L^{2}}+\f{1}{2}\|\Lambda^{ \alpha}u _{t}^{(m+n-1)}\|_{L^{2}}^{2} \\
\leq&C\sum_{j>0}\|u_{t}^{(j)}\|
^{\f{2(4(m+n-1)\alpha+6\alpha-5)}{4j\alpha+4\alpha-5}}
_{L^{2}}
+C\sum_{j>0}\|\Lambda^{\alpha}u_{t}^{(m+n-1-j)} \|_{L^{2}}^{\f{2(4(m+n-1)\alpha+6\alpha-5)}{4(m+n-1-j)\alpha+6\alpha-5}}
\\&+\varepsilon\sum_{j>0}\|\Lambda^{\alpha}u_{t}^{(m+n-1-j)} \|_{L^{2}}^{\f{4\alpha(2j-1)}{  4(m+n-1-j)\alpha+6\alpha-5}}
\|\Lambda^{2\alpha}u_{t}^{(m+n-1-j)} \|^{2}_{L^{2}}\\
\leq&C\sum_{j=1}^{m+n-2}F_{j,0}
^{\f{ 4(m+n-1)\alpha+6\alpha-5 }{4j\alpha+4\alpha-5}}
+C\| u _{t}^{(m+n-1)}\| _{L^{2}}^{\f{ 2[4(m+n-1)\alpha+6\alpha-5] }{4(m+n-1)\alpha+4\alpha-5}}
+C\sum_{j=0}^{m+n-2}F_{j,1} ^{\f{ 4(m+n-1)\alpha+6\alpha-5 }{4j\alpha+6\alpha-5}}
\\&+\eta\sum_{j=0}^{m+n-2}F_{j,1}^{\f{  4\alpha(m+n-1-j)-2\alpha }{  4 j\alpha+6\alpha-5}}G_{j,1}.
\ea
\ee
By virtue of the definition of $F_{ m+n-1,0}$  in \eqref{4.61v1}, we have
\be\ba
\sum_{j=1}^{m+n-2}F_{j,0}
^{\f{ 4(m+n-1)\alpha+6\alpha-5 }{4j\alpha+4\alpha-5}}\leq C F_{ m+n-1,0}^{\f{4(m+n-1)\alpha+6\alpha-5}{4(m+n-1)\alpha+4 \alpha-5}},\\
\sum_{j=1}^{m+n-3}F_{j,1}
^{\f{ 4(m+n-1)\alpha+6\alpha-5 }{4j\alpha+6\alpha-5}}\leq C F_{ m+n-1,0}^{\f{4(m+n-1)\alpha+6\alpha-5}{4(m+n-1)\alpha+4 \alpha-5}},\\
F_{ m+n-2,1}^{\f{ 4(m+n-2)\alpha+10\alpha-5}
{4(m+n-2)\alpha+6 \alpha-5}}\leq C F_{ m+n-1,0}^{\f{4(m+n-1)\alpha+6\alpha-5}{4(m+n-1)\alpha+4 \alpha-5}}.
 \ea\ee
Plugging this into \eqref{4.62}, we arrive at
 \be\ba\label{4.58}
&\f{d}{dt}\| u _{t}^{(m+n-1)}\|^{2}_{L^{2}}+\f{1}{2}\|\Lambda^{ \alpha}u _{t}^{(m+n-1)}\|_{L^{2}}^{2} \\
\leq& C F_{ m+n-1,0}^{\f{4(m+n-1)\alpha+6\alpha-5}{4(m+n-1)\alpha+4 \alpha-5}}
+C\| u _{t}^{(m+n-1)}\| _{L^{2}}^{\f{ 2[4(m+n-1)\alpha+6\alpha-5] }{4(m+n-1)\alpha+4\alpha-5}}
 +\eta\sum_{j=0}^{m+n-2}F_{j,1}^{\f{  4\alpha(m+n-1-j)-2\alpha }{  4 j\alpha+6\alpha-5}}G_{j,1}.
\ea
\ee
For $j+1\leq m+n-2$,   by  the induction hypotheses \eqref{4.41v1}, we observe that
$$\ba
&\f{d}{dt}F_{j,1}+G_{j,1}\leq CF_{j,1}^{\f{4j\alpha+ 8 \alpha-5}{4j\alpha+6 \alpha-5}},
\ea$$
which implies that
\be\label{4.65}
c_{j,1}\f{d}{dt}
F_{j,1}^{\f{4(m+n-1)\alpha+4 \alpha-5}{4j\alpha+6\alpha-5}}+F_{j,1}^{\f{4(m+n-1-j)\alpha-2 \alpha }{4j\alpha+6 \alpha-5}}G_{j,1}\leq CF_{j,1}^{\f{4(m+n-1)\alpha + 6 \alpha-5}{4j\alpha+6 \alpha-5}}\leq C F_{ m+n-1,0}^{\f{4(m+n-1)\alpha+6\alpha-5}{4(m+n-1)\alpha+4 \alpha-5}}.
\ee
For $j= m+n-2$, we apply \eqref{4.60} in Case 3 to deduce that
\be\ba
&\f{d}{dt}F_{ m+n-2,1}+G_{ m+n-2,1}\leq CF_{ m+n-2,1}^{\f{ 4(m+n-2)\alpha+8\alpha-5 }
{4(m+n-2)\alpha+6 \alpha-5}},
\ea\ee
which leads to
\be\ba\label{4.66}
&\f{d}{dt}F_{ m+n-2,1}+F_{m+n-2,1}^{\f{   2\alpha }{  4 (m+n-2)\alpha+6\alpha-5}}G_{ m+n-2,1}\leq CF_{ m+n-2,1}^{\f{ 4(m+n-2)\alpha+10\alpha-5}
{4(m+n-2)\alpha+6 \alpha-5}}\leq C F_{ m+n-1,0}^{\f{4(m+n-1)\alpha+6\alpha-5}{4(m+n-1)\alpha+4 \alpha-5}}.
\ea\ee
For $j+0\leq m+n-2$,  in light of the induction hypotheses    \eqref{4.13}, we know that
$$\ba
&\f{d}{dt}F_{j,0}+G_{j,0}\leq CF_{j,0}^{\f{4j\alpha+ 6 \alpha-5}{4j\alpha+4 \alpha-5}},
\ea$$
which gives that
\be\ba\label{4.63}
&c_{j,0}\f{d}{dt}F_{j,0}^{\f{4(m+n-1)\alpha+4 \alpha-5}{4j\alpha+4\alpha-5}}+F_{j,0}^{\f{4(m+n-1-j)\alpha }{4j\alpha+4\alpha-5}}G_{j,0}\leq CF_{j,0}^{\f{4(m+n-1 )\alpha+ 6 \alpha-5}{4j\alpha+4 \alpha-5}}\leq C F_{ m+n-1,0}^{\f{4(m+n-1)\alpha+6\alpha-5}{4(m+n-1)\alpha+4 \alpha-5}}.
\ea\ee
Collecting the estimates \eqref{4.58}, \eqref{4.65}, \eqref{4.66} and \eqref{4.63}, we end up with
\be\ba\label{4.64}
& \f{d}{dt}\B[\| u _{t}^{(m+n-1)}\|^{2}_{L^{2}} +\sum_{\ell=0}^{m+n-2}
c_{\ell,0}F_{\ell,0}^{\f{4(m+n-1)\alpha+4 \alpha-5}{4\ell\alpha+4\alpha-5}}+\sum_{\ell=0}^{m+n-2}
c_{\ell,1}F_{\ell,1}^{\f{4(m+n-1)\alpha+4 \alpha-5}{4\ell\alpha+6\alpha-5}} \B]\\&+\sum_{j=0}^{m+n-3}F_{j,1}^{\f{4(m+n-1-j)\alpha-2 \alpha }{4j\alpha+6 \alpha-5}}G_{j,1}+F_{m+n-2,1}^{\f{   2\alpha }{  4 (m+n-2)\alpha+6\alpha-5}}G_{ m+n-2,1} +\f12\|\Lambda^{ \alpha}u _{t}^{(m+n-1)}\|_{L^{2}}^{2} \\
\leq& C F_{ m+n-1,0}^{\f{4(m+n-1)\alpha+6\alpha-5}{4(m+n-1)\alpha+4 \alpha-5}}
+C\| u _{t}^{(m+n-1)}\| _{L^{2}}^{\f{ 2[4(m+n-1)\alpha+6\alpha-5] }{4(m+n-1)\alpha+4\alpha-5}}
 +\eta\sum_{j=0}^{m+n-2}F_{j,1}^{\f{  4\alpha(m+n-1-j)-2\alpha }{  4 j\alpha+6\alpha-5}}G_{j,1},
\ea
\ee
that is
\be\ba
\f{d}{dt}F_{ m+n-1,0}+G_{ m+n-1,0}\leq C F_{ m+n-1,0}^{\f{4(m+n-1)\alpha+6\alpha-5}{4(m+n-1)\alpha+4 \alpha-5}}.
\ea
\ee
Here, we omit the terms based on $F_{j,0}^{\f{4(m+n-1-j)\alpha }{4j\alpha+4\alpha-5}}G_{j,0}$ on the left hand side of above inequality.
It is a position to apply
Lemma \ref{hofflemma} and  \eqref{4.61} to get
\be\ba
G_{ m+n-1,0}\in L^{\f{6\alpha-5}{4(m+n-1)\alpha+6\alpha-5}},
\|\Lambda^{ \alpha}u _{t}^{(m+n-1)}\|_{L^{2}}^{2}\in L^{\f{6\alpha-5}{4(m+n-1)\alpha+6\alpha-5}}.
\ea
\ee
We therefore prove \eqref{4.13} with $p+q=k-1= m+n -1$ for Case 4.

Summarily, no mater in which case, we always have
\be\ba
\|\Lambda^{ (m+n-i)\alpha}u _{t}^{(i)}\|_{L^{2}}^{2}\in L^{\f{6\alpha-5}{4i \alpha+2(m+n-i)\alpha+4\alpha-5}},
\ea
\ee

$$\|\Lambda^{n}u^{(m)}_{t}\|_{L^{2}}  \leq  C  \| u^{(m)}_{t}\|_{L^{2}}^{ \f{2\alpha-1}{2\alpha}} \|\Lambda^{2n\alpha }u^{(m)}_{t}\|_{L^{2}}^{\f{1}{2\alpha}}.$$
$$\ba
&\int_{0}^{T}\|\Lambda^{n}u^{(m)}_{t}\|^{\f{2(6\alpha-5)}{4m\alpha+2n+4\alpha-5}}_{L^{2}} dt \\
\leq&  C  \int_{0}^{T}\| u^{(m)}_{t}\|_{L^{2}}^{ \f{(2\alpha-1)(6\alpha-5)}{ \alpha(4m\alpha+2n+4\alpha-5)}} \|\Lambda^{2n\alpha }u^{(m)}_{t}\|_{L^{2}}^{\f{ (6\alpha-5)}{ (4m\alpha+2n+4\alpha-5) \alpha}}dt\\
\leq&  C \B(\int_{0}^{T}\| u^{(m)}_{t}\|_{L^{2}}^{ \f{2(6\alpha-5)}{  4m\alpha +4\alpha-5  }} dt\B) ^{\f{2\alpha(4m\alpha+2n+4\alpha-5)}{(2\alpha-1)(4m\alpha +4\alpha-5 )}} \B(\int_{0}^{T}  \|\Lambda^{2n\alpha }u^{(m)}_{t}\|_{L^{2}}^{\f{2(6\alpha-5)}{ (4m\alpha+4n\alpha+4\alpha-5) }}dt \B)^{1-\f{2\alpha(4m\alpha+2n+4\alpha-5)}{(2\alpha-1)(4m\alpha +4\alpha-5 )}}.\ea$$

The proof of this theorem is completed.
\end{proof}
Finally, we present the proof of Corollary \ref{coro} via interpolation.
\begin{proof}[Proof of Corollary \ref{coro}]
Assume for a while we have proved that
\be\label{while}
\int_{0}^{T}\|\Lambda^{k}u\|
^{\f{1}{k +1}}_{L^{\infty}}dt<\infty.
\ee
(1) The interpolation inequality and Sobolev embedding theorem give
$$\ba
\| u\|_{L^{q}}\leq \| u\|_{L^{6}}^{\f6q}\| u\|^{1-\f6q}_{L^{\infty}}\leq C\| \nabla u\|_{L^{2}}^{\f6q}\| u\|^{1-\f6q}_{L^{\infty}}, ~6\leq q \leq\infty.
\ea$$
After performing a time integration, one conclude by the H\"older inequality that
$$\ba
\int_{0}^{T}\| u\|_{L^{q}}^{\f{q}{q-3}}dt \leq C \B(\int_{0}^{T}\|\nabla u\|_{L^{2}}^{2}dt\B)^{\f{3}{q-3}}\B( \int_{0}^{T} \|u\|_{L^{\infty}}dt\B)^{1-\f{3}{q-3}}.
\ea$$
As a result, by \eqref{while},  we get $  u   \in L^{\f{q}{q-3}}(0,T;L^{q} (\mathbb{R}^{3})), ~6\leq q\leq\infty.$

(2)
We derive from the interpolation inequality that
$$\ba
\|\Lambda^{k}u\|_{L^{q}}\leq \|\Lambda^{k}u\|_{L^{2}}^{\f2q}\|\Lambda^{k}u\|^{1-\f2q}_{L^{\infty}}, ~2\leq q \leq\infty.
\ea$$
Integrating the above inequality in time and using the H\"older inequality, one checks that
\be\ba
\int_{0}^{T}\|\Lambda^{k}u\|_{L^{q}}^{\f{q}{kq+q-3}}dt\leq & \int_{0}^{T}\|\Lambda^{k}u\|_{L^{2}}^{\f{2}{kq+q-3}}\|\Lambda^{k}u\|^{\f{q-2}{kq+q-3}}_{L^{\infty}}dt
\\
\leq& \B(\int_{0}^{T}\|\Lambda^{k}u\|_{L^{2}}^{\f{2}{2k -1}}dt\B)^{\f{2k -1}{kq+q-3}}\B(\int_{0}^{T}\|\Lambda^{k}u\|
^{\f{1}{k +1}}_{L^{\infty}}dt
\B) ^{1-\f{2k -1}{kq+q-3}}.
\ea\ee
Consequently, with the help of \eqref{1.1} with $\alpha=1$  in Theorem \ref{the1.1} and  \eqref{while},  we find
$\Lambda^{k}u   \in L^{\f{q}{q(k+1)-3}}(0,T;L^{q} (\mathbb{R}^{3})),~k\geq1,~2\leq q\leq\infty.$\\
It remains to prove
the inequality \eqref{while} we have assumed. Indeed, invoking the Gagliardo-Nirenberg inequality and Sobolev embedding, we have
$$\|\Lambda^{k}u\|
^{ }_{L^{\infty}}\leq C\| u\|
^{\f{1}{2(k+1)} }_{L^{6}} \|\Lambda^{k+2}u\|
^{ \f{2k+1}{2(k+1)} }_{L^{2}}\leq  C\| \nabla u\|
^{\f{1}{2(k+1)} }_{L^{2}} \|\Lambda^{k+2}u\|
^{ \f{2k+1}{2(k+1)} }_{L^{2}}.$$
By integrating it with respect to time, we observe that
$$\ba
\int_{0}^{T}\|\Lambda^{k}u\|
^{\f{1}{k+1} }_{L^{\infty}}dt\leq&   C\int_{0}^{T}\| \nabla u\|
^{\f{1}{2(k+1)^{2}} }_{L^{2}} \|\Lambda^{k+2}u\|
^{ \f{2k+1}{2(k+1)^{2}} }_{L^{2}}dt\\
\leq&   C\B(\int_{0}^{T}\| \nabla u\|
^{2 }_{L^{2}}dt\B)^{\f{1}{4(k+1)^{2}}} \B(\int_{0}^{T}\|\Lambda^{k+2}u\|
^{ \f{2 }{2 k+3} }_{L^{2}}dt\B)^{1-\f{1}{4(k+1)^{2}}}.
\ea$$
It is a position to apply \eqref{1.1} with $\alpha=1$  in Theorem \ref{the1.1} to get assertion \eqref{while}.
This proves the corollary.
\end{proof}

\end{document}